\documentclass[11pt,a4paper]{article}
\usepackage{amsmath,amssymb,amsthm,graphicx}
\usepackage{titlesec,hyperref}
\usepackage{mathtools,url,enumitem}
\usepackage[parfill]{parskip}		%
\usepackage{caption}
\usepackage{subcaption}
\usepackage{setspace}
\usepackage[noadjust]{cite}
\usepackage{psfrag}
\usepackage{hyperref}
\hypersetup{
    colorlinks=true, %
    linkcolor=black, %
    urlcolor=black, %
    citecolor=black,
    linktoc=all %
}

\allowdisplaybreaks

\newcommand{\G}{\mathcal{G}}

\newcommand{\C}{\mathbb{C}}
\newcommand{\CC}{\widehat{\mathbb{C}}}

\renewcommand{\H}{\mathbb{H}}

\newcommand{\R}{\mathbb{R}}

\newcommand{\g}{\gamma}

\newcommand{\sgn}{\textnormal{sgn}}

\newtheorem{thm}{Theorem}[section]
\newtheorem{prop}[thm]{Proposition}
\newtheorem{lem}[thm]{Lemma}
\newtheorem{cor}[thm]{Corollary}

\newtheorem{df}[thm]{Definition}

\theoremstyle{definition}
\newtheorem{remark}[thm]{Remark}

\newtheorem{definition}[thm]{Definition}
\newtheorem{example}[thm]{Example}

\usepackage[font=normal]{caption}
\usepackage{floatrow}

\usepackage{mathtools}

\usepackage[dvipsnames]{xcolor}

\global\long\def\ii{\mathfrak{i}}
\global\long\def\ee{\mathrm{e}}

\newcommand{\abs}[1]{\left\lvert #1 \right \rvert}

\newcommand{\mc}[1]{\mathcal{#1}}
\newcommand{\m}[1]{\mathbb{#1}}

\renewcommand\Re{\operatorname{Re}}
\renewcommand\Im{\operatorname{Im}}

\def\PSL{\operatorname{PSL}}

\def\SLE{\operatorname{SLE}}

\def\mob{\mathrm{M\ddot ob}}

\def\a{\alpha}
\def\b{\beta}
\def\g{\gamma}
\def\G{\Gamma}

\def\t{\theta}

\def\l{\lambda}
\def\k{\kappa}

\def\O{\Omega}
\def\vare{\varepsilon}
\def\HH{{\mathbb H}}
\def\Chat{\widehat{\m{C}}}

\def\dd{\mathrm{d}}

\newcommand{\ad}[1]{\overline{#1}}

 \newcommand{\splus}{{\scriptstyle +}}
 \newcommand{\sminus}{{\scriptstyle -}}

 \def \1{\mathbf{1}}

\def\Id{\operatorname{Id}}

 \newcommand{\weld}{\varphi}
 \def\S{\Sigma}
 \newcommand{\marking}{\omega}

\newcommand{\cover}{\Pi}

\begin{document}

\title{Piecewise geodesic Jordan curves II: Loewner energy, projective structures, and accessory parameters}
\author{Mario Bonk\thanks{\protect\url{mbonk@math.ucla.edu} University of California, Los Angeles, USA} \qquad
Janne Junnila\thanks{\protect\url{janne@junnila.me} KTH Royal Institute of Technology, Stockholm, Sweden} \qquad Steffen Rohde\thanks{\protect\url{rohde@math.washington.edu} 
University of Washington, USA}
\qquad Yilin Wang\thanks{\protect\url{yilin.wang@math.ethz.ch} ETH Z\"urich, Switzerland}} \maketitle

\begin{abstract}
In this paper, we consider Jordan curves on the Riemann sphere passing through $n \ge 3$ given points. We show that in each relative isotopy class of such curves, there exists a unique curve that minimizes the Loewner energy. These curves have the property that each arc between two consecutive points is a hyperbolic geodesic in the domain bounded by the other arcs.
This geodesic property lets us define a complex projective structure whose holonomy lies in $\PSL(2,\R)$. We show that the quadratic differential comparing this projective structure to the trivial projective structure on the sphere has simple poles whose residues (accessory parameters) are given by the Wirtinger derivatives of the minimal Loewner energy. This is reminiscent of Polyakov's conjecture for Fuchsian projective structures, proven by Takhtajan and Zograf. Finally, we show that the projective structures we obtain are related to Fuchsian projective structures through $\pi$-grafting.
\end{abstract}

\section{Introduction}

\subsection{Main results}
Let $n \ge 3$ be an integer. We call a Jordan curve $\gamma $ on the Riemann sphere $\Chat= \m C \cup \{\infty\}$ through distinct points $z_1,...,z_n$  \emph{piecewise geodesic} if
$\gamma\smallsetminus\{z_1,...,z_n\}=\gamma_1\cup \gamma_2\cup \ldots\cup \gamma_n$ where $\gamma_j$ is a hyperbolic geodesic in 
$\Chat\smallsetminus (\gamma\smallsetminus \gamma_j)$ for all
$j=1,\ldots,n$. %
This class of Jordan curves was introduced and investigated from geometric function theoretic viewpoints in \cite{MRW1} (see also \cite{W1,RW,peltola_wang,bonk2021canonical}). Specifically, the conformal welding of these curves was characterized, the role of smoothness was elucidated, and the local expressions of the uniformization map of its complement were explicitly computed. 

In the present paper, which is self-contained, we investigate these curves from the Loewner theory perspectives and their relation to (complex) projective structures on the $n$-punctured sphere. 
Throughout this article, we call a function or a Jordan curve \emph{smooth} if it is once continuously differentiable.  Smooth piecewise geodesic curves arise naturally as minimizers of a certain functional related to Loewner theory, namely Loewner energy (see also \cite{Mesikepp_min,W_opt} for closely related energy minimizing problems). In Section \ref{sec:preliminaries}, we review the basic properties of Loewner energy and show (Lemma~\ref{lem:loop_geodesic_minimizer}) that every relative isotopy class of curves through points $z_1,...,z_n$ contains at least one Loewner energy minimizer. Furthermore, each minimizer is piecewise geodesic. We also describe in Section \ref{sec:Teichmuller} that these isotopy classes are in bijection with the Teichm\"uller space $T_{0,n}$ of the $n$-punctured sphere.

The first main result (Theorem \ref{thm:uniqueness}) of this paper is the {\it uniqueness} of the smooth piecewise geodesic curve in its relative isotopy class. Hence, the unique smooth piecewise geodesic curve coincides with the energy minimizer. The proof relies on potential theory and does not refer to the Loewner theory. 
In other words, we obtain a parametrization of $T_{0,n}$ by the family of smooth piecewise geodesic Jordan curves.
The welding homeomorphism associated with a Jordan curve is a homeomorphism of  $\widehat{\m R} = \m R \cup \{\infty\}$.
It was known that the welding homeomorphisms of piecewise geodesic curves are precisely the piecewise M\"obius weldings (see \cite{MRW1} and Section \ref{sub:geodesic}), hence Theorem \ref{thm:uniqueness} also yields a simple parametrization of
$T_{0,n}$ by smooth piecewise M\"obius weldings as described in Section \ref{sec:C1_lambda}. Theorem \ref{thm:differentiability}, the main result of Section \ref{sec:C1_lambda}, shows that this parametrization provides a diffeomorphism between the space of smooth piecewise M\"obius maps %
and $T_{0,n}$.

In Section \ref{sec:accessory}, we associate a projective structure $Z_\g$ on the punctured sphere with each smooth piecewise geodesic curve $\g$. The projective charts are given by the uniformizing conformal maps in the complement of the curve, which can be extended across each arc $\g_j$ thanks to the piecewise geodesic property. 
Theorem~\ref{thm:uniqueness} then implies that we obtain a well-defined section of projective structures over the Teichm\"uller space $T_{0,n}$ given by the unique smooth piecewise geodesic curve in each relative isotopy class.

Furthermore, we observe that the holonomy of $Z_\g$ around each puncture is parabolic (Theorem~\ref{thm:Z_gamma}) and that the quadratic differential comparing $Z_\g$ to the trivial projective structure induced from $\Chat$  
has simple poles. The main result of Section \ref{sec:accessory} is Theorem \ref{thm:Jordan_accessory}, identifying the residues of the poles at the punctures $z_k$ 
as the Wirtinger derivatives $\partial_{z_k}$ of the %
Loewner energy of the smooth piecewise 
geodesic curve through the punctures. %
This result is very similar to the relation between the accessory parameters of the Fuchsian projective structure and the classical Liouville action, first conjectured by Polyakov and proved by Takhtajan and Zograf in \cite[Thm.\,1]{TZ1}. Our proof relies on the explicit computation of the conformal map of geodesic pairs from \cite{MRW1} (recalled in Section \ref{sub:pair-welding}) and the differentiability results of Section \ref{sec:C1_lambda}.

The final Section \ref{sec:grafting} studies the relation between the Fuchsian projective structures on an $n$-punctured sphere and the projective structures $Z_\g$ described above. We show that they are related by a $\pi$-angle grafting along hyperbolic geodesics connecting the punctures in the $n$-punctured sphere (Proposition~\ref{prop:central_geod}). An immediate corollary of our main results (see Corollary \ref{cor:bijection}) shows that the $\pi$-angle grafting induces a diffeomorphism of $T_{0,n}$.
This corollary is similar to the results \cite{scannell2002grafting,McMullen_complex} for closed surfaces and general measured laminations, showing that the grafting map gives a self-homeomorphim of the Teichm\"uller space. We also mention that for closed surfaces, Goldman \cite{goldman1987projective} showed that every projective structure with $\PSL(2,\R)$-holonomy can be obtained by grafting the Fuchsian projective structure along suitable multiple $\pi$-graftings.

\subsection{Comments and discussions}
 Although we will not use these links for our results, let us end this introduction by mentioning, on a heuristic level, some connections to Schramm--Loewner evolution (SLE) and the analytic Langlands correspondence.

Recall that the $\SLE_\kappa$ loop measure $\mu_\k$, where $0<\k\le 4,$ is a one-parameter family of measures on Jordan curves invariant under M\"obius transformations of the sphere \cite{zhan2020sleloop,Werner_loop,Benoist_loop}. Its distribution is governed by the Loewner energy $I^L$. With well-chosen neighborhoods, 
$$\frac{\mu_\kappa \{\text{Loops in an $\vare$-neighborhood of } \g\}}{\mu_\kappa \{\text{Loops in an $\vare$-neighborhood of } S^1\}} \xrightarrow[]{\vare\to 0} \exp \left(\frac{c(\k)}{24} I^L(\g)\right),$$
where $c(\k) = (3\k - 8) (6-\k)/(2\k) \sim_{\k \to 0} -24/\k$. See \cite{carfagnini2023onsager} for the precise statement.
Thus, heuristically, as $\k \to 0$ (the semiclassical limit), SLE$_\k$ loops concentrate around energy-minimizers.  The piecewise geodesic curves we study can be viewed as $\SLE_{0+}$-loops conditioned to pass through the given points in the given isotopy class. 
We mention that the classical Liouville action discussed above is also related to the semiclassical limit of Liouville conformal field theory, a theory of random surfaces, see, e.g., \cite{LRV}.

A common feature of the projective structures we consider (Fuchsian projective structure and $Z_\g$) is that they all have $\PSL(2,\R)$-holonomy. Such projective structures are also related to the real opers forming one side of the analytic Langlands correspondence \cite{gaiotto2024quantum,Etingof,Teschner18}. 
In fact, given a projective structure, its developing map $h$ can be recovered from the Schwarzian derivative by taking the ratio of two independent solutions of the differential equation $\partial_{zz} u + t(z) u = 0$ where $t(z) = \mathcal{S}[h](z)/2$, having at most double order poles.
The corresponding differential operators $\partial_{zz} + t$ are called \emph{real opers} if the resulting developing map has holonomies in  $\PSL(2,\R)$.

In this context, it was conjectured in \cite[Appx.\,A.1]{gaiotto2024quantum} that real opers are classified topologically by the real locus of their developing map (\emph{singularity lines} in the language of \cite{gaiotto2024quantum}). 
Our result on the existence and uniqueness of smooth piecewise geodesic curves confirms this conjecture in the case of the $n$-punctured sphere, when the real locus is assumed to be a smooth Jordan curve passing through the punctures. %

As mentioned in \cite{MRW1}, without the smoothness assumption, the Jordan curve can have logarithmic spirals at the punctures. We can expect uniqueness only after the spiraling rates at the vertices have been specified. This explains why, in order to fully characterize the real opers with a given topological data of the curve passing through the points, one needs to specify the double poles of $t$ which determine the spiralling rate.
However, we think that our uniqueness proof, which studies the real locus of the projective structure in a very direct manner, can be adapted to show the uniqueness of other projective structures with $\PSL(2,\R)$-holonomy as well, and we intend to continue studying this question in the future.

\section{Preliminaries}\label{sec:preliminaries}
\subsection{Teichm\"uller space of $n$-punctured sphere} \label{sec:Teichmuller}

Recall that two Jordan curves $\gamma^0, \gamma^1$ are {\it isotopic relative to $z_1,...,z_n \in \CC$} if there is a continuous family of Jordan curves $\gamma^t$ through $z_1,...,z_n$ 
beginning with $\gamma^0$ and ending with $\gamma^1.$
We write 
$\mc L(z_1, \ldots, z_n; \g)$ for the isotopy
class of Jordan curves containing (a representative curve) $\g$ which goes through $z_1, \ldots, z_n$ in this order. 

In Section \ref{sub:geodesic} we will show that each isotopy class of curves contains a minimizer of the Loewner energy, and in the following section that the minimizer is unique. 
In the present section, we briefly discuss the role of the Teichm\"uller space $T_{0,n}$ of the $n-$punctured sphere %
and describe how each isotopy class of Jordan curves through $n$ given points can be viewed as a point in $T_{0,n}$. %

To fix a definition of $T_{0,n}$, we fix the base Riemann surface 
$$\S_0 = \Chat\smallsetminus\{-(n-3),\ldots, -2,-1,0,1,\infty\}$$ 
and think of the extended real line $\widehat\R$ as a 
Jordan curve
on the sphere through the $n$ punctures $-(n-3),-(n-2),\ldots,1,\infty$ of $\S_0.$
A {\it marked $n$-punctured sphere} is a pair %
$(\S,\marking)$ where $\S = \CC \smallsetminus \{z_1, \ldots, z_{n-3}, 0, 1, \infty\}$ and $\marking:\S_0\to \S$ a quasiconformal homeomorphism sending punctures to punctures (more precisely, sending $-(n-3)$ to $z_1$,..., and $0,1,\infty$ to $0,1,\infty$ in this order). Two marked surfaces $(\S,\marking)$ and $(\widehat \S, \widehat \marking)$ are {\it equivalent} if the 
quasiconformal homeomorphism $\widehat \marking\circ \marking^{-1}:\S \to \widehat \S$ is isotopic to a conformal map while fixing the marked points, hence isotopic to the identity map. Then $T_{0,n}$ is the set of equivalence classes of marked
$n$-punctured spheres.

Now we turn our attention to Jordan curves passing through $n$ points  $z_1,...,z_n$ in this order. By composing with the M\"obius transformation that sends $z_{n-2}, z_{n-1}$ and $z_n$ to $0,1,\infty$, 
we may restrict our attention to Jordan curves $\gamma$ passing through points $z_1,...,z_{n-3},0,1,\infty$
(in this order). As will become clear later, we are only interested in quasicircles, slightly simplifying the following presentation.
Given a quasicircle $\gamma$ through $z_1,...,z_{n-3},0,1,\infty,$ there is a quasiconformal homeomorphism $\marking$ of $\Chat$
mapping $\widehat\R$ to $\gamma$ and
$-(n-3),\ldots, -2,-1,0,1,\infty$ to $z_1,...,z_{n-3},0,1,\infty$ (in this order).
Any two such quasiconformal maps $\marking,\widehat \marking$ are isotopic, and we therefore obtain a map
$$F: \{((z_1,...,z_{n-3},0,1,\infty),\gamma)\}\to T_{0,n}$$
by setting $F\big((z_1,...,z_{n-3},0,1,\infty),\gamma \big) = (\S,\marking)$ with $\S=\Chat\smallsetminus\{z_1,...,z_{n-3},0,1,\infty\}.$
Furthermore, if $\gamma$ and $\widehat \gamma$ are isotopic quasicircles through the same set of points, and if $\marking$ and $\widehat \marking$ are quasiconformal 
homeomorphisms mapping $\widehat\R$ to $\gamma$ and $\widehat \gamma$ as above, then $\marking$ and $\widehat \marking$ are again isotopic so that the map $F$
descends to a map from the set of isotopy classes $\{\mc L(z_1,...,z_{n-3},0,1,\infty; \gamma )\}$ to $T_{0,n}$.

Conversely, every point of $T_{0,n}$ is represented by some quasiconformal map $\marking:\S_0\to \Chat \smallsetminus \{z_1,...,z_{n-3},0,1,\infty\}$ fixing $0, 1, \infty$, and the points $z_1,...,z_{n-3}$ are 
independent of the choice of the representative $\marking$.  %
Moreover, the isotopy class of the Jordan curve $\marking(\widehat\R)$ is independent of the choice of $\marking$, so that every point of $T_{0,n}$ corresponds to a unique
point $\mc L(z_1,...,z_{n-3},0,1,\infty;\gamma)$. We summarize our discussion as follows.

\begin{prop}
The space of relative isotopy classes of quasicircles marked by $n$ ordered points $z_1,...,z_{n-3},0,1,\infty$ 
can be naturally identified with $T_{0,n}$ as above.
\end{prop}

The marking $\marking$ also provides a marking of the elements of the fundamental group $\pi_1(\S)$. We write $\eta_1 \in \pi_1 (\S_0)$ for the loop starting from $\ii \in \CC$ and going around $-(n-3)$ by crossing first the interval $(-\infty, -(n-3))$ then $(-(n-3), -(n-2))$ before going back to $\ii$; $\eta_2$ for the loop from $\ii$ and going around $-(n-2)$ by crossing first $(-(n-3),-(n-2))$ then $(-(n-2), -(n-1))$ before going back to $\ii$, etc.
The family $\{\eta_1, \ldots, \eta_n\}$ is a generator of $\pi_1 = \pi_1 (\S_0)$ with the only relation $\eta_1 \eta_2 \cdots \eta_n = \Id$.

\begin{figure}[H]
\centering
  \includegraphics[width=0.9\linewidth]{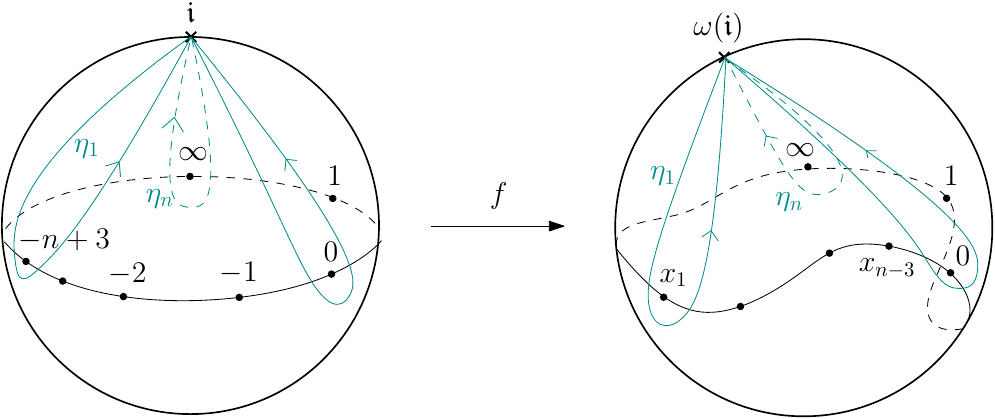}
  \caption{\label{fig:fundamental} The marking of $\pi_1$ and a set of generators.}
\end{figure}

\subsection{Marked projective structure and Schwarzian}

The space of marked projective structures on a surface fibers over the Teichm\"uller space. We refer the readers to the survey paper by Dumas \cite{Dumas_projective} for the general theory and describe here the relevant notions for the $n$-punctured sphere. 

Recall that the group $$\PSL(2,\C) = \left \{ A
= \begin{psmallmatrix} a & b \\ c & d
\end{psmallmatrix} \colon a,b,c,d \in \C, \, ad - bc = 1\right\}_{/A\sim - A}$$
acts on $\Chat = \C \m P^1$ by %
M\"obius transformations $z \mapsto (az +b)/(cz+d)$.  
And the subgroup $$\mob (\widehat\R) = \PSL(2,\m R) =  \left \{ A
= \begin{psmallmatrix}  a & b \\ c & d
\end{psmallmatrix} \colon a,b,c,d \in \R, \, ad - bc = 1\right\}_{/A\sim - A}$$
acts on $\m H = \{z \in \CC \colon \Im (z) > 0\}$, $\widehat \R$, and $\m H^* = \{z \in \CC \colon \Im (z) < 0\}$ in a similar way.

A {\it complex projective structure} $Z$ (or {\it projective structure} 
for short) on $\S$ is a maximal atlas of projective local charts $(U, \varphi_U)_{U \in I}$, meaning that the coordinate changes $\varphi_V \circ \varphi_U^{-1}$ belong to $\PSL(2,\C)$, see Example \ref{ex:projective structures} below.

A {\it marked projective structure} on the $n$-punctured sphere is a triple $(\S,Z,\marking)$,  where $(\S,\marking) \in T_{0,n}$ and such that the projective structure $Z$ on $\S$ induces the complex structure on $\S$. %
We write $\mc P(\S,\marking)$ for the set of projective structures over $(\S,\marking) \in T_{0,n}$.

Each marked projective structure $Z$ defines a developing map and a holonomy map as follows.
Here and in the sequel, let $\widetilde \S$ be the universal cover of $\S$ and $\cover : \widetilde \S \to \S$ a covering map. Each projective structure on $\S$ lifts to a projective structure $\widetilde Z$ on $\widetilde \S$.
A {\it developing map} of $Z$ is a holomorphic immersion $h: \widetilde \S \to \Chat$ such that restricting to a small open set gives a local projective chart of $\widetilde Z$. 
Two developing maps $h_1, h_2$ are related by $h_1 = A \circ h_2$ on $\widetilde \S$ for some $A \in \PSL(2, \C)$.

The group of deck transformations consists of biholomorphic functions $\a: \widetilde \S \to \widetilde \S$ such that $\cover \circ \a  = \cover$. This group is isomorphic to $\pi_1 (\S)$ via the map sending $\a$ to the loop $\cover ([\a(x), x])$, where $x \in \widetilde \S$ is a lift of $\marking(\ii)$ and $[\a(x),x]$ is any continuous path in $\widetilde \S$ connecting $x$ and $\a(x)$. Since $\pi_1 (\S)$ is further identified with the fundamental group $\pi_1$ of $\S_0$ through the marking $\marking$, we use $\pi_1$ to also denote the group of deck transformations.

Since $h \circ \a$ is also a developing map, there is $A_\a \in \PSL(2,\C)$ such that
$$A_\a \circ h= h \circ \a, \qquad \forall \a \in \pi_1.$$
The group homomorphism $\rho: \a \mapsto A_\a$, $\pi_1 \to \PSL(2,\m C)$ is called the {\it holonomy representation} of $\pi_1$ associated with the projective structure $Z$.
Post-composing $h$ by $A \in \PSL(2,\C)$ conjugates the holonomy representation:
\begin{equation}\label{eq:conjugate_A}
    (h, \rho) \mapsto (A \circ h, \rho^A), \qquad \text{where } \rho^A : \a \mapsto A \circ A_\a \circ A^{-1}.
\end{equation}
Summarizing, the projective structure $Z$ uniquely determines the development-holonomy pair $(h,\rho)$ up to the conjugation \eqref{eq:conjugate_A}.  Conversely, $h$ determines $Z$ as restrictions of $h \circ \cover^{-1}$  to small neighborhood in $\S$ are the projective local charts.

\begin{example}\label{ex:projective structures}
\begin{itemize}
    \item {\it Trivial projective structure} on $\S = \Chat\smallsetminus\{z_1,...,z_{n-3},0,1,\infty\}$: The developing map $h_0 = \cover$, and $\rho_0$ is the trivial representation $\pi_1 \to \{\Id\}$. Hence the projective local charts are given by the identity map. 
    \item {\it Fuchsian projective structure} $Z_F$ on $\S$.  The developing map $h_{F}$ is taken to be a uniformizing map $\widetilde \S \to \m H$. The holonomy representation $\rho_F$ takes values in $\PSL(2,\R)$. 
    Moreover, $\rho_F(\eta_k)$ is parabolic (recall that $\{\eta_k\}_{k = 1,\ldots n}$ is the family of generators of $\pi_1$ defined above). The local charts are $h_F \circ \cover^{-1}$ restricted to small enough open sets. Any other choice of the biholomorphic map $\widetilde \S \to \m H$ is obtained from $h_F$ by post-composing an element of $\PSL(2,\R)$, which results in the same projective structure.
\end{itemize}
\end{example}

The set $\mc P(\S,\marking)$ can be parametrized using Schwarzian derivatives which compare a projective structure to the trivial projective structure: If $Z \in \mc P(\S,\marking)$ has the development-holonomy pair $(h, \rho)$, then $h \circ h_0^{-1} = h\circ \cover^{-1}$ is a multi-valued function from $\S \subset \CC \to \CC$.
In fact, for a small open set of $\S$, two choices of lift $\cover^{-1}$ and  $\widehat \cover^{-1}$ are related by $\widehat \cover^{-1} = \a \circ \cover^{-1}$ for some $\a \in \pi_1 : \widetilde \S \to \widetilde \S$. And 
\begin{equation}\label{eq:compare_Z}
h \circ \widehat \cover^{-1}  = h \circ \a \circ \cover^{-1} = \rho (\a) \circ h \circ \cover^{-1}.
\end{equation}
Recall that the {\it Schwarzian derivative} of a holomorphic immersion, defined as
$$\mc S [\varphi] = \frac{\varphi'''}{\varphi'} - \frac{3}{2}\left(\frac{\varphi''}{\varphi'}\right)^2, $$
satisfies
$\mc S[A] \equiv 0$ for all $A \in \PSL(2,\C)$
and the chain rule
\begin{equation}\label{eq:chain_S}
    \mc S[\varphi\circ \psi] = \mc S[\varphi] \circ \psi (\psi')^2 + \mc S[\psi].
\end{equation}

Informally, the Schwarzian derivative measures the local deviation of a holomorphic immersion from an M\"obius transformation.
In particular, \eqref{eq:compare_Z} shows that although $h \circ h_0^{-1}$ is multi-valued, its Schwarzian derivative is well-defined on $\S$:
$$\mc S[h \circ \widehat \cover^{-1}] = \mc S[h \circ \cover^{-1}].$$

Any other developing map of $Z$ is obtained from the transformation \eqref{eq:conjugate_A}, which gives the same Schwarzian derivative $S [h \circ \cover^{-1}]$. %
\begin{lem}\label{lem:quadratic_diff}
Each projective structure $Z \in \mc P(\S,\marking)$ uniquely defines a quadratic differential $q = \mc S[h \circ \cover^{-1}] \, \dd z^2$ on $\S$, where $h$ is any developing map of $Z$. If $Z_1$ and $Z_2 \in \mc P(\S,\marking)$ define the same quadratic differential, then %
$Z_1=Z_2.$
\end{lem}
\begin{proof}
The function $q$ is a quadratic differential since if $A \in \PSL(2,\C)$ and $\S = A (\S_1)$, then the corresponding quadratic differential $q_1$ on $\S_1$ transforms like
\begin{equation}\label{eq:quadratic_differential}
q_1 = \mc S[h \circ \cover^{-1} \circ A] \,\dd z^2 = q(A) (A')^2 \dd z^2 = A^* q
\end{equation}
by the chain rule \eqref{eq:chain_S}.

For the second claim, let $h_j$ be a developing map of $Z_j$, then the assumption $\mc S[h_1 \circ \cover^{-1}] = \mc S[h_2 \circ \cover^{-1}]$ shows that there exists $A \in \PSL(2,\C)$ such that $h_2 = A \circ h_1$. Therefore, the restrictions of $h_1$ and $h_2$ to small open sets are compatible projective charts and $Z_1 = Z_2$.
\end{proof}

\subsection{Loewner energy of curves and loops}

In this section, we recall the definition and basic facts about Loewner energy.

We say that $\g$ is a {\it simple curve} in $(D;a,b)$, if $D$ is a simply connected domain, $a, b$ are two distinct prime ends of  $D$, and $\g$ has a continuous and injective parametrization $(0,T) \to D$ such that $\g(t) \to a$ as $t \to 0$ and $\g(t) \to c \in D \cup \{b\}$ as $t \to T$. If $c = b$ then we say $\g$ is a {\it (simple) chord} in $(D;a,b)$.

Simple curves in $(\m H; 0,\infty)$ can be encoded into a chordal driving function as follows (this procedure is called the {\it Loewner transform}). 
We first parameterize the curve by the half-plane capacity. More precisely, $\g$ is continuously parametrized by $[0,T]$, where $T \in (0, \infty]$ with $\g(0) = 0$, 
such that
the unique conformal map $g_t$ from $\HH \smallsetminus \gamma[0,t]$ onto $\HH$ with the expansion at infinity $g_t (z)=  z + o(1)$ satisfies
\begin{equation}\label{eq:hydro}
g_t (z)= z + \frac{2t}{ z} + o\left( \frac{1}{z}\right), \qquad \forall t \in (0,T).
\end{equation}

The coefficient $2t$ is the {\it half-plane capacity} of $\g[0,t]$, and $2T$ is the {\it total capacity} of $\g$.
The map $g_t$ can be extended by continuity to the boundary point $\gamma(t)$ and that the real-valued function $W(t) := g_t ( \gamma(t))$ is continuous with $W(0) = 0$ (i.e., $W \in C^0[0,T]$). This function $W$ is called the {\it driving function} of $\g$. The Loewner transform satisfies the following properties:

\begin{itemize}
    \item (\emph{Additivity}) %
    Fix $s > 0$, the driving function generating $\left(\ad{g_s(\g[s,t])} - W(s)\right)_{t \in [0,T-s]}$ is $ t \mapsto W(s+t) - W(s)$.
\item (\emph{Scaling}) Fix $\l >0$, the driving function generating the scaled and capacity-reparameterized curve $\left(\l \g(\l^{-2}t)\right)_{t \in [0,\l^2 T]}$ is  $t\mapsto \l W (\l^{-2}t)$.
\end{itemize}

\begin{df} \label{df:chordal_LE}
Let $\g$ be a simple curve in $(D;a,b)$. Let $\phi$ be a conformal map from $D$ onto $\H$ 
with $\phi(a) = 0$, $\phi(b) = \infty$, $T$ be the half-plane capacity of $\phi(\g)$.
 The {\it chordal Loewner energy} of $\g$ in  $(D;a,b)$
 is 
$$ I_{D;a,b} (\g) := I_{\m H;0,\infty} (\phi (\g) ) : =  \frac{1}{2} \int_0^T \dot{W}(t)^2 \,\dd t \in [0,\infty] $$
if the chordal driving function $W$  of $\phi (\g)$ is absolutely continuous, and $\infty$ otherwise. 
\end{df}

If $\g$ makes it all the way to $b$ (namely, if $\g$ is a chord) and if $I_{D;a,b} (\g) < \infty$ then $T = \infty$ (namely, $\g$ has infinite total capacity). 

Note that the definition does not depend on the conformal map $\phi$ chosen since, from the scaling property of the driving function, we have that for all $ \l >0$,
$I_{\H, 0, \infty}(\g) = I_{\H, 0, \infty}( \l \g)$.
\begin{remark}\label{rem:min_energy}
    The Loewner energy of a chord attains its minimum $0$ if and only if the driving function is constantly $0$, namely $\phi (\g)$ is $\ii\R_+$ which is equivalent to the curve $\g$ being the {\it hyperbolic geodesic} of $D$ from $a$ to $b$. We note that a hyperbolic geodesic is an analytic curve. 
\end{remark}

The additivity of the Loewner transform immediately implies the following lemma.

\begin{lem}[Additivity of Loewner energy]\label{lem:add_energy}
For any continuous parametrization $\g: [0,1] \to D$ with $\g(0) = a$ and $\g (1) = b$,
and for any $t < 1$, we have
$$I_{D;a,b}(\g[0,1]) = I_{D;a,b}(\g[0,t]) + I_{D \smallsetminus \g[0,t];\g(t),b}(\g[t,1]).$$
\end{lem}
We note that although $\g$ is not necessarily parametrized by capacity,  $I_{D;a,b}(\g[0,1])$ is understood as in Definition~\ref{df:chordal_LE} where we start with parametrizing $\phi(\g)$ by capacity. Thus $I_{D;a,b}(\g[0,1])$ is trivially invariant under increasing reparametrization. 
It was also shown in \cite{W1} that the chordal Loewner energy does not depend on the orientation of the curve. Therefore, we may view the curves as being unoriented.

Let us next recall the definition of Loewner energy for Jordan curves, introduced in \cite{RW}.
Let $\g: [0,1] \to \Chat : = \m C \cup \{\infty\}$ be a continuously parametrized Jordan curve with the marked point $\g (0) = \g(1)$. 
For every $\vare>0$, $\g [\vare, 1]$ is a chord connecting $\g(\vare)$ to $\g (1)$ in the simply connected domain $\Chat \smallsetminus \g[0, \vare]$.

\begin{df}
The {\it Loewner energy} of $\g$ rooted at $\g(0)$ is 
$$I^L(\g, \g(0)): = \lim_{\vare \to 0} I_{\Chat \smallsetminus \g[0, \vare]; \g(\vare), \g(0)} (\g[\vare, 1]).$$
\end{df}

The third and fourth authors proved the following result.
\begin{thm}[\!\! \cite{RW}] \label{thm:intro_root_invariance}
The Loewner energy of the rooted Jordan curve does not depend on the root chosen.
\end{thm}

Hence, the Loewner energy is a quantity on the family of unrooted Jordan curves that is invariant under M\"obius transformations. The root independence is further explained by the following equivalent expression of the Loewner energy involving only two conformal maps.

\begin{thm}[{\cite[Thm.\,1.3]{W2}}]\label{thm:intro_equiv_energy_WP}
Let $\g$ be a bounded Jordan curve. Let $\O$ and $\O^*$ denote respectively the bounded and unbounded components of $\C \smallsetminus \g$. Let $f$ be a conformal map from $\m D$ onto $\O$ and $g$ a conformal map from $\m D^*$ onto $\O^*$ fixing $\infty$,
we have the identity
       \begin{equation} \label{eq_disk_energy}
   I^L(\gamma) =  \mc D_{\m D} (\log \abs{f'}) + \mc D_{\m D^*} (\log \abs{g'})+ 4 \log \abs{f'(0)} - 4 \log \abs{g'(\infty)},
 \end{equation}
 where $g'(\infty):=\lim_{z\to \infty} g'(z)$ and $\mc D_D (u) = \int_D \abs{\nabla u}^2 /\pi \, |\dd z|^2$ denotes the Dirichlet energy  of $u$ on domain $D$.
\end{thm}

This result also identifies the family of Jordan curves with finite energy with the family of {\it Weil--Petersson quasicircles} introduced in \cite{TT06} where the right-hand side of \eqref{eq_disk_energy} is called \emph{universal Liouville action}.
This family of Jordan curves has been the topic of very active research thanks to a large number of equivalent descriptions \cite{Bishop,shen13,W3,cui00,johansson2021strong,johansson2023coulomb,epstein_yilin}.

\subsection{Energy minimizers and the geodesic property}\label{sub:geodesic}

In this section, we recall the basic properties of energy-minimizing chords and loops. These results are mostly contained in \cite{MRW1}.

\begin{lem}[{\cite[Prop.\,3.1]{W1}}] \label{lem:8ln} Let $\t \in (0,\pi)$.
The Loewner energy of any chord in $(\m H; 0, \infty)$ passing through $z = \ee^{\ii \t}$ is greater than or equal to $-8 \log \sin (\t)$. The minimum is attained by a unique chord $\g^z$. %
\end{lem}

The energy minimizer turns out to be the unique $C^1$ smooth chord with the following geodesic property. It will be an essential building block in studying the Jordan curves with the geodesic property.

\begin{df} Let  $z \in D$ be an interior point of $D$.
We define a {\it geodesic pair} in $(D;a,b;z)$ to be a chord $\g$ in  $(D;a,b)$ such that $\g = \g_1 \cup \g_2 \cup \{z\}$, $\g_1$ is the geodesic in $(D \smallsetminus \g_2; a, z)$ and $\g_2$ is the geodesic in $(D \smallsetminus \g_1; z, b)$.
\end{df}

In the companion paper \cite{MRW1}, we classified all geodesic pairs in $(\m H; 0,\infty; z)$. In particular, \cite[Thm.\,3.9]{MRW1} shows that there is a unique $C^1$ smooth geodesic pair in $(\m H; 0,\infty; z)$; %
\cite[Cor.\,3.10]{MRW1} shows that if a geodesic pair is not $C^1$, then it has a logarithmic spiral near $z$
which has necessarily infinite energy.

\begin{lem}\label{lem:geodesic_property_pair}
The energy minimizing chord $\g^z$ is a geodesic pair in $(\m H; 0,\infty; z)$ and is $C^{1,1-\vare}$ for all $\vare > 0$. 
\end{lem}
\begin{proof} 
We write $\g^z = \g_1 \cup \g_2 \cup \{z\}$ such that $\g_1$ has the end points $a$ and $z$ and $\g_2$ has the end points $z$ and $b$. Lemma~\ref{lem:add_energy} shows that 
 $$I_{D;a,b}(\g^z)= I_{D;a,b}(\g_1) + I_{D\smallsetminus \g_1;z,b}(\g_2).
 $$
 Since $\g^z$ is minimizing the energy among all chords passing through $z$, $I_{D\smallsetminus \g_1;z,b}(\g_2) = 0$. Otherwise, replacing $\g_2$ by the geodesic in $(D \smallsetminus \g_1; z, b)$ decreases the energy. 
 Similarly, considering $\g$ as a chord in $(D;b,a)$ shows that $\g_1$ is the hyperbolic geodesic in  $(D \smallsetminus \g_2; z, a)$.
 
 The regularity of $\g^z$ follows from the fact that the driving function of $\g^z$ is $C^{1,1/2}$ regular. In fact, the driving function of $\g^z$ is studied in \cite{W1}, see e.g. Equation (3.2) there. The main theorem of \cite{Wong} shows that $\g^z$ is weakly $C^{1,1}$ (which implies that $\g^z$ is $C^{1,1-\vare}$) regular.
 \end{proof}

\begin{cor}
The energy minimizer $\g^z$ is the unique $C^1$ smooth geodesic pair in $(\m H; 0, \infty; z)$.
\end{cor}

\begin{remark}
    The regularity of a $C^1$ geodesic pair at its vertex can also be seen from Lemma~\ref{lem:ftheta-asymptotics} below.
\end{remark}

Let us now move to the loop case.
A Jordan curve $\gamma\subset \Chat$ through distinct points $z_1,...,z_n$
is called {\it piecewise geodesic} if
$\gamma\smallsetminus\{z_1,...,z_n\}=\gamma_1\cup \gamma_2\cup \ldots\cup \gamma_n$ where $\gamma_j$ is a hyperbolic geodesic in 
$(\Chat\smallsetminus \gamma)\cup \gamma_j$ for all
$j=1,\ldots,n$. We also say that $\g$ has the {\it geodesic property}. Such curves arise naturally as Loewner energy minimizers:

\begin{lem}\label{lem:loop_geodesic_minimizer}
There exists at least one minimizer of the Loewner energy in each isotopy class $\mc{L}(z_1, z_2,\cdots, z_n; \tau)$.  Any minimizer has the geodesic property and is $C^{1,1-\vare}$ regular for all $\vare > 0$. 
\end{lem}
\begin{proof} 
In \cite[Prop.\,2.13]{RW}, it was shown that there exists at least one minimizer of the Loewner energy passing through a collection of distinct points $z_1, z_2,\cdots, z_n$ (without the constraint on the isotopy class). 
The same proof can be easily modified to show that each isotopy class $\mc{L}(z_1, z_2,\cdots, z_n;\tau)$ 
 contains at least one energy minimizer by taking an energy-minimizing sequence within the isotopy class and passing it to a subsequential limit. 
For the geodesic property, it follows from the root-invariance (Theorem~\ref{thm:intro_root_invariance}, Theorem~\ref{thm:intro_equiv_energy_WP}) and the additivity Lemma~\ref{lem:add_energy}  that 
$\g_{k-1} \cup \g_{k} \cup \{z_k\}$ has to be an energy minimizing chord through $z_k$ in $\big((\Chat\smallsetminus \gamma)\cup  \g_{k-1} \cup \g_{k} \cup \{z_k\}; z_{k-1}, z_{k+1}\big)$ for every $k$. Lemma~\ref{lem:geodesic_property_pair} then shows that it is a geodesic pair with $C^{1,1-\vare}$ regularity which implies the geodesic property of $\g$ with the same regularity.
\end{proof}

\bigskip 

\begin{remark}
One may wonder if the Jordan curve with the geodesic property is a concatenation of hyperbolic geodesics in the $n$-punctured sphere connecting the punctures. Lemma 4.2 of \cite{MRW1} shows that they are very different. In fact, if $\g$ %
is a Jordan curve with the geodesic property, and if each arc $\g_k$ is a 
hyperbolic geodesic in  $\Sigma: = \Chat \smallsetminus \{z_1, \ldots, z_n\}$ between $z_k$ and $z_{k+1}$, then $\g$ is a circle. In Section~\ref{sec:grafting}, we show that the curve with the geodesic property and the curve formed by hyperbolic geodesics in the $n$-punctured sphere are related by a $\pi$-angle grafting.
\end{remark}

\bigskip 

The geodesic property can easily be characterized in terms of conformal welding.

More precisely, let $\g$ %
be a Jordan curve through points $z_1,\ldots,z_n$ (in this order),  %
let $\O$ be the connected component of $\Chat \smallsetminus \g$ where $(z_1, \ldots, z_n)$ goes around $\O$ counterclockwise, and denote $\O^*$ the other connected component.
Let $f$ and $g$ be respectively a conformal map from $\m H$ onto $\O$ and from $\m H^*$ onto $\O^*$. Then the {\it welding homeomorphism} of $\g$ is given by $\weld := g^{-1} \circ f|_{\widehat \R}.$ It is determined by $\g$ only up to pre- and post-composition with M\"obius automorphisms in $\PSL(2,\m R)$. Let $x_k=f^{-1}(z_k)$ and $y_k=g^{-1}(z_k)=\weld(x_k).$ A homeomorphism $\weld$ of $\widehat \R$ is called {\it piecewise M\"obius} if $\widehat \R$ can be decomposed into finitely many intervals such that on each interval, $\weld$ is (the restriction of) a M\"obius automorphisms.

\begin{thm}[{\cite[Cor.\,4.1 and Lem.\,4.12]{MRW1}}]\label{thm:piecewise_mob}
A Jordan curve $\g$ has the geodesic property if and only if the welding homeomorphism $w$ is piecewise M\"obius. %
Moreover, $\g$ is $C^1$ smooth if and only if the welding $\weld$ is $C^1$ smooth.
\end{thm}

\subsection{Geodesic pairs and conformal welding}\label{sub:pair-welding}

In this section, we study the conformal welding of geodesic pairs.
These auxiliary results will be used in Sections \ref{sec:C1_lambda} and \ref{sec:accessory} to pinpoint the local behavior of the welding maps of piecewise geodesic loops.

Let us consider $\theta \in (0,\pi)$ and the geodesic pair $\gamma$ in $(\m D; \ee^{-\ii \theta}, \ee^{\ii \theta}; 0)$.
We write $\gamma = \gamma_1 \cup \gamma_2 \cup \{0\}$ where $\gamma_1$ connects $\ee^{\ii \theta}$ to $0$ and $\gamma_2$ connects $\ee^{-\ii \theta}$ to $0$.
As shown in \cite{MRW1}, the map
\[h_\theta(z) = G_{\pi/2 - \theta}(\ii z) = \frac{1}{2}\left(\ii z + \frac{1}{\ii z}\right) - \ii \cos(\theta) \log(\ii z)\]
maps the two connected components of $\m D \smallsetminus \gamma$ to upper and lower half-planes.
Here the branch of $\log(\ii z)$ is chosen so that %
the imaginary part is $\pi/2$ for $z \in (0,1)$ and the branch cut is along $\gamma_1$.
See Figure~\ref{fig:G_maps} for an illustration.
   
\begin{figure}
\centering
  \includegraphics[width=0.8\linewidth]{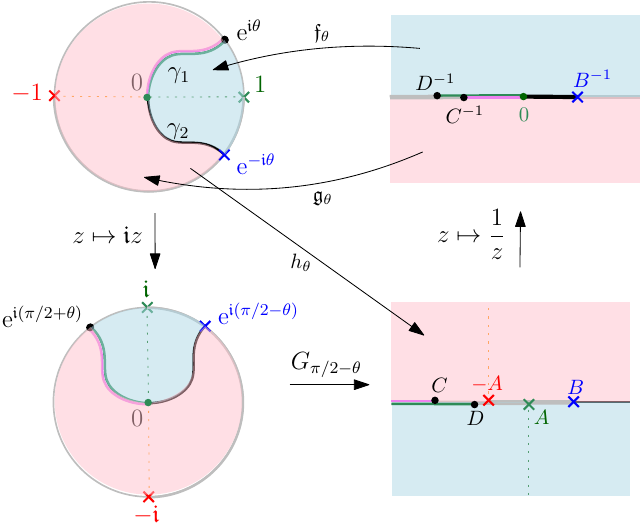}
  \caption{\label{fig:G_maps} Illustration of the uniformizing conformal maps associated with a standard geodesic pair in $(\m D; \ee^{\ii \theta}, \ee^{-\ii \theta}; 0)$. It was computed in  \cite{MRW1} (see Figure 1 there) that $A = (\pi/2) \cos (\theta)$, $B = \sin (\theta) + (\pi /2 -\theta) \cos (\theta)$, $C = - 2A - B$, $D = 2A - B$.}
\end{figure}

The map $h_\theta$ is analytic in $\m D \smallsetminus \gamma_1$ and satisfies $h_\theta(0) = \infty$, $h_\theta(1) = -h_\theta(-1) = A := (\pi/2) \cos(\theta)$ and $h(\ee^{-\ii\theta}) = B := \sin(\theta) + (\pi/2 - \theta)\cos(\theta)$.
The point $\ee^{\ii\theta}$ gets mapped to either $C = -2A-B$ or $D = 2A-B$, depending on the side of the approach.

Finally, we define the conformal maps $\mathfrak f_\theta(z) = h_\theta^{-1}(1/z)$ (for $z \in \H$) and $\mathfrak g_\theta(z) = h_\theta^{-1}(1/z)$ (for $z \in \H^*$).
They both map $0$ to $0$ and extend continuously to the boundary.

\begin{lem}\label{lem:pair-welding-homeo}
Let $\varphi_\theta \colon [1/D,1/B] \to [1/C, 1/B]$ be the welding homeomorphism $\varphi_\theta(x) = (\mathfrak g_\theta^{-1} \circ \mathfrak f_\theta)(x)$ for the geodesic pair.
Then
\[\varphi_\theta(x) = \begin{cases}
    \dfrac{x}{(C-D)x + 1}, & \text{if } x \in [1/D,0], \\
    x, & \text{if } x \in [0,1/B].
\end{cases}\]
Moreover, the jump 
\begin{equation}\label{eq:jump_Def}
\lambda[\varphi_\theta](0) := \varphi_{\theta}''/\varphi_{\theta}'(0\splus) - \varphi_{\theta}''/\varphi_{\theta}'(0\sminus)
\end{equation}
of the pre-Schwarzian
$\varphi_{\theta}''/\varphi_{\theta}'$ 
at $0$ equals $-4\pi \cos(\theta)$.
\end{lem}

\begin{proof}
    It was computed in \cite{MRW1} that the welding homeomorphism $\phi$ for $h_\theta^{-1}$ (i.e. without composing with $1/z$, which also swaps upper and lower half-planes) is given by $\phi(x) = x + D-C$ when $x < C$ and $\phi(x) = x$ when $x > B$.
    The first claim follows since we have $\varphi_\theta(x) = 1/\phi^{-1}(1/x)$.
    We note that the pre-Schwarzian of $\varphi_{\theta}$ is
    \[\dfrac{\varphi_\theta''(x)}{\varphi_\theta'(x)} = \begin{cases} \dfrac{-2(C-D)}{(C-D)x+1} & \text{when } x \in [1/D,0), \\ 0 & \text{when } x \in (0,1/B].\end{cases}\]
    In particular,\[\lambda[\varphi_\theta](0) = 2(C-D) = -8A = -4\pi \cos(\theta).\qedhere\]
\end{proof}

We will also need precise asymptotics of $\mathfrak f_\theta$ and $\mathfrak g_\theta$ at $0$.

\begin{lem}\label{lem:ftheta-asymptotics}
We have the following asymptotic expansions as $z \to 0$:
\begin{align*}
\mathfrak f_\theta(z) & = \frac{1}{2\ii}z - \frac{1}{2} \cos(\theta) z^2 \log(z) + O(z^2),\\
\mathfrak f_\theta'(z) & = \frac{1}{2\ii} - \cos(\theta) z \log(z) + O(z), \\
\partial_\theta \mathfrak f_\theta(z) & = \frac{1}{2} \sin(\theta) z^2 \log(z) + O(z^2).
\end{align*}
The same asymptotics hold with $\mathfrak f_\theta$ replaced by $\mathfrak g_\theta$.
\end{lem}

\begin{proof}
    Let us write $\mathfrak f_\theta(z) = w$.
    By definition, we have
    \[z = \mathfrak f_\theta^{-1}(w) = \frac{1}{h_\theta(w)} = \frac{2\ii w}{1 + 2 \cos(\theta) w \log(\ii w) - w^2}.\]
    Note that since $w \to 0$ as $z \to 0$, this implies that $|z| \ge |w|$ for $z$ small enough, so that
    \[w = O(z).\]
    We next write the above equation as
    \[z = 2 \ii w - \frac{4 \ii \cos(\theta) w^2 \log(\ii w) - 2 \ii w^3}{1 + 2 \cos(\theta) w \log(\ii w) - w^2},\]
    implying that
    \[w = \frac{z}{2\ii} + \frac{2 \cos(\theta) w^2 \log(\ii w) - w^3}{1 + 2 \cos(\theta) w \log(\ii w) - w^2} = \frac{z}{2 \ii} + 2 \cos(\theta) w^2 \log(\ii w) + O(w^2).\]
    In particular, $w = O(z)$ now implies that $w = z/(2\ii) + O(z^2 \log(z)) = z/(2 \ii) + o(z)$, so that $w^2 = -z^2/4 + O(z^3 \log(z))$
    and $\log(\ii w) = \log(z/2 + o(z)) = \log(|z/2||1 + o(1)|) + O(1) = \log(|z|) + O(1) = \log(z) + O(1).$ This gives us the first asymptotics in the statement,
    \[\mathfrak f_\theta(z) = w = \frac{z}{2 \ii} - \frac{1}{2} \cos(\theta) z^2 \log(z) + O(z^2).\]
    Next, we note that
    \[(\mathfrak f_\theta^{-1})'(w) = \frac{2 \ii - 4 \ii \cos(\theta) w + 2 \ii w^2}{(1 + 2 \cos(\theta) w \log(\ii w) - w^2)^2},\]
    so that
    \[\mathfrak f_\theta'(z) = \frac{1}{(\mathfrak f_\theta^{-1})'(w)} = \frac{1 + 4 \cos(\theta) w \log(\ii w) + O(w)}{2 \ii + O(w)}.\]
    Applying the asymptotics for $w$ obtained above, we get
    \[\mathfrak f_\theta'(z) = \frac{1}{2\ii} - \cos(\theta) z \log(z) + O(z)\]
    as wanted.
    Finally, note that differentiating the identity $\mathfrak f_\theta(\mathfrak f_\theta^{-1}(w)) = w$ with respect to $\theta$ on both sides gives us
    \[\partial_\theta \mathfrak f_\theta(\mathfrak f_\theta^{-1}(w)) + \mathfrak f_\theta'(\mathfrak f_\theta^{-1}(w)) \partial_\theta [\mathfrak f_\theta^{-1}](w) = 0,\]
    so that
    \[\partial_\theta \mathfrak f_\theta(z) = - \mathfrak f_\theta'(z) \partial_\theta [\mathfrak f_\theta^{-1}](w) = - \mathfrak f_\theta'(z) \left(\frac{4\ii \sin(\theta) w^2 \log(\ii w)}{(1 + 2 \cos(\theta) w \log(\ii w) - w^2)^2}\right).\]
    Again, applying the already obtained asymptotics for $w$ gives us
    \[\partial_\theta \mathfrak f_\theta(z) = \frac{1}{2} \sin(\theta) z^2 \log(z) + O(z^2).\]
    The same proof goes through verbatim for $\mathfrak g_\theta$ as well.
\end{proof}

\section{Uniqueness of piecewise geodesic $C^1$ Jordan curves}\label{sec:uniqueness}

In this section, we prove that every isotopy class $\mc L(z_1, \dots, z_{n}; \tau)$ contains a unique piecewise geodesic $C^1$ smooth Jordan curve.
It follows in particular that the minimizer of Loewner energy obtained in Lemma~\ref{lem:loop_geodesic_minimizer} is unique.

\begin{thm}\label{thm:uniqueness}
Let $\gamma, \widehat{\gamma} \in \mc L(z_1,\dots,z_n;\tau)$ be two isotopic $C^1$ smooth Jordan curves. If $\gamma$ and $\widehat{\gamma}$ are piecewise geodesic, then $\gamma = \widehat{\gamma}$.
\end{thm}

Roughly speaking, our proof is based on constructing for each boundary arc $\gamma_k$ a suitable `distance' or `potential' function $h_k$ on the sphere and then using it to measure how far the corresponding arc $\widehat \gamma_k$ of $\widehat \gamma$ lies from $\gamma_k$, with the goal being to show that $h_k|_{\widehat \gamma_k} = 0$.
As the hyperbolic geodesic $\gamma_k$ consists of exactly those points which get mapped to the real axis when the domain $(\CC \smallsetminus \gamma) \cup \gamma_k$ is mapped conformally to the strip $\{z \in \C : \Im(z) \in (-1,1)\}$ and the endpoints of $\g_k$ mapped to $\pm \infty$, a natural candidate for $h_k$ is the imaginary part of this mapping.
When trying to implement such a strategy, one quickly notices that $\widehat \gamma_k$ may get further away from $\gamma_k$ by intersecting other arcs of $\gamma$. It is, therefore, natural to shift our viewpoint from the sphere to the universal cover of $\Sigma$.
For this purpose, let us fix some notation.

We let $%
\cover \colon \m D \to \Sigma = \Chat \smallsetminus \{z_1, \ldots, z_n\}$ be a holomorphic universal covering. The set $\cover^{-1}(\gamma)$ is the union of countably many arcs $\ell_\alpha \subset \m D$, each of them is a lift of  $\gamma_k$ for some $k \in \{1,\dots,n\}$. Here, $\alpha \in I$ runs over some abstract countable index set $I$. We denote the two connected components of $\Chat \smallsetminus \g$ as $\O_+$ and $\O_-$.

Each $\ell_\alpha$ in turn will be a hyperbolic geodesic in a domain $U_\alpha \subset \m D$, corresponding to a lift of $\Omega_+ \cup \Omega_- \cup \cover(\ell_\alpha)$. Moreover, each $U_\alpha$ naturally splits into subdomains $U_\alpha^+$ and $U_\alpha^-$ so that $\cover(U_\alpha^{\pm}) = \Omega_\pm$. %
Note that any pair of domains $U_\alpha^+$ and $U_\beta^+$ (resp. $U_\alpha^-$ and $U_\beta^-$) are equivalent modulo deck transformations, and the domains $\{U_\alpha^\pm : \alpha \in I\}$ form a tiling of $\m D$.

Let next $H \colon [0,1] \times \Sigma \to \Sigma$ be an isotopy of $\Sigma$ such that $H(0,\gamma) = \gamma$ and $H(1,\gamma) = \widehat{\gamma}$. 
Then $H$ admits a lift $\widetilde{H} \colon [0,1] \times \m D \to \m D$ to the universal cover of $\Sigma$ such that $\widetilde{H}(t,\cdot)|_{\partial \m D} = \Id$ for all $t \in [0,1]$, see \cite[V.1.4]{lehto2012univalent}. 
We use $\widetilde H(1, \cdot)$ to define the arcs and domains: $\widehat \ell_\alpha = \widetilde{H}(1,\ell_\alpha)$, $\widehat U_\alpha = \widetilde{H}(1,U_\alpha)$ and $\widehat U_\alpha^{\pm} = \widetilde{H}(1,U_\alpha^{\pm})$ corresponding to the curve $\widehat \gamma$. Note in particular that $\ell_\alpha$ and $\widehat \ell_\alpha$ share the same endpoints.

We will next define, for each fixed $\alpha \in I$, a potential function $h_\alpha$ as hinted above. 
We start by constructing a rooted tree $\mc T_\alpha$ on $I,$ see Figure \ref{fig:covering1}. 
The root $\alpha$ is connected by an edge to each of the $2n-2$ labels $\mu$ that correspond to boundary arcs $\ell_\mu$ of $U_\alpha.$ Two distinct labels $\nu, \mu \in I\setminus\{\alpha\}$ are connected by an edge
if either $U_\nu^+ = U_\mu^+$ or $U_\nu^- = U_\mu^-$, and if furthermore
$\ell_\nu$ separates $\ell_\mu$ and $\ell_\alpha$ in $\m D$ (or if $\ell_\mu$ separates $\ell_\nu$ and $\ell_\alpha$). In other words, there is an edge between $\nu$ and $\mu$ if $\ell_\nu$ and $\ell_\mu$ are both part of the boundary of a common tile, and if one is further away from $\ell_\alpha$ than the other.
We let $d (\cdot, \cdot)$ denote the graph distance in this tree and also let
\[\sigma(\alpha,\nu) = \begin{cases} 0, & \text{if } \alpha = \nu \\ 1, & \text{if the path from } \alpha \text{ to } \nu \text{ goes through some } \mu \neq \alpha \text{ with } \ell_\mu \subset \partial U_\alpha^+ \\
-1, & \text{if the path from } \alpha \text{ to } \nu \text{ goes through some } \mu \neq \alpha \text{ with } \ell_\mu \subset \partial U_\alpha^-.
\end{cases}\]
Finally, we let $h_\alpha \colon \m D \to \R$ be the function which is harmonic outside of $\cover^{-1}(\gamma)$ and which on any given arc $\ell_\nu$ ($\nu \in I$) takes the constant value $\sigma(\alpha,\nu) d(\alpha, \nu)$.

\begin{figure}
    \centering
    \includegraphics[width=.95\linewidth]{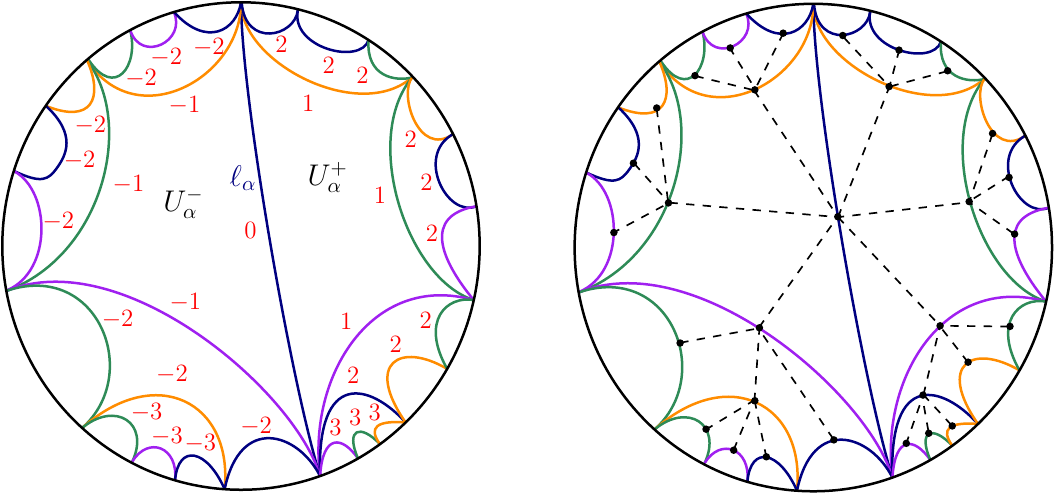}
    \caption{The covering space picture. Left: The values of $h_\alpha$ on the arcs $\ell_\beta$ are written in red. Arcs with the same color are equivalent under the deck transformations. Right: The rooted tree $\mc T_\a$ is drawn with dashed lines.}
    \label{fig:covering1}
\end{figure}

As stated in the beginning, we aim to show that $h_\alpha$ vanishes on $\widehat{\ell}_\alpha$. We begin by showing that it is bounded.

\begin{lem}\label{lem:rho_bounded}
There exists a constant $C > 0$ such that $|h_\alpha(z)| \le C$ for all $\alpha \in I$ and $z \in \widehat \ell_\alpha$.
\end{lem}

\begin{proof}
Note that since $\gamma$ and $\widehat \gamma$ are $C^1$ smooth arcs, $\widehat \gamma_k$ cannot hit both $\gamma_k$ and $\gamma_{k+1}$ (resp. $\gamma_{k-1}$) infinitely many times as it approaches $z_{k+1}$ (resp. $z_k$).
It follows that if $\widehat \ell_\alpha$ is a lift of $\widehat \gamma_k$, it intersects at most finitely many domains $U_\beta$ and hence $|h_\alpha|$ is bounded on $\widehat \ell_\alpha$. The bound obtained 
only depends on $k.$
\end{proof}

Let us next define the auxiliary function $\widetilde \rho \colon \m D \to \R$ by first setting %
$\widetilde \rho|_{\widehat{\ell}_\alpha} = |h_\alpha|$ on every arc $\widehat \ell_{\alpha}$ and then extending harmonically inside each tile $\widehat U_\alpha^\pm$.
Note that $\widetilde \rho$ is invariant under the deck transformations of the covering $\cover$, and hence there exists a continuous and bounded function $\rho \colon \Sigma \to \R$ such that $\widetilde \rho = \rho \circ \cover$.
The following could be considered the main lemma of our proof as it implies that $\rho$ has to be a constant.

\begin{lem}\label{lem:rho_subharmonic}
The function $\rho$ is subharmonic.
\end{lem}

Before proving the lemma, let us state and prove a couple of auxiliary results. We begin with the following elementary criterion for subharmonicity.

\begin{lem}\label{lem:subharmonicity_criterion}
Suppose that the upper-semicontinuous function $f \colon \m D \to \R$ satisfies the following property: For every $z \in \m D$, there exists a subharmonic function $u$ defined in a neighborhood of $z$ such that $u \le f$ and $u(z) = f(z)$. Then $f$ is subharmonic.
\end{lem}

\begin{proof}
It is straightforward to check that the sub-mean value property of $f$ at $z$ follows from that of $u$.
\end{proof}

A second auxiliary result we will use is the subharmonicity of $|h_\alpha|$ itself.

\begin{lem}\label{lem:halpha_subharmonic}
The function $h_\alpha$ is harmonic in $U_\alpha$ and $|h_\alpha|$ is subharmonic in $\m D$.
\end{lem}

\begin{proof}[Proof of Lemma~\ref{lem:halpha_subharmonic}]
The harmonicity of $h_\alpha$ in $U_\alpha$ follows from the assumption that $\ell_\alpha$ is a hyperbolic geodesic in $U_\alpha$:
Indeed, $h_\alpha$ coincides with the imaginary part of a conformal map from $U_\alpha$ to the strip $\{z \in \C : \Im(z) \in (-1,1)\}$ mapping the two boundary components $\partial U_\alpha^\pm \smallsetminus \ell_\alpha$ to the horizontal lines $\Im(w) = \pm 1$, respectively.
Clearly, in the strip, the hyperbolic geodesic between $-\infty$ and $+\infty$ is the real line $\Im(w) = 0$.

To show that $|h_\alpha|$ is subharmonic in $\m D$, we will verify the condition of Lemma~\ref{lem:subharmonicity_criterion}.

The function $|h_\alpha|$ is clearly continuous. 
As $h_\alpha$ is harmonic and has a constant sign in each tile $U_\beta^{\pm}$, it is enough to verify the subharmonicity at $z \in \ell_\beta$ for some $\beta \in I\setminus \{\alpha\}$.
Moreover, by symmetry, we may assume that $\sigma(\alpha,\beta) = 1$.
Consider the harmonic function $u$ in the domain $U_\beta$ which takes value $d(\alpha,\beta) + 1$ on every boundary arc $\ell_\nu \subset \partial U_\beta$ such that %
$d(\alpha,\nu)=d(\alpha,\beta) + 1$,
and $d(\alpha,\beta) - 1$ on every other boundary arc.
Clearly $u$ equals $|h_\alpha|$ on $\ell_\beta$ as well as on the boundary arcs $\ell_\nu$ for which %
$d(\alpha,\nu)=d(\alpha,\beta) + 1$.
If $\mu$ is the parent in $\mc T_\alpha$ of $\beta$ we also have $u(w) = |h_\alpha(z)|$ for $w \in \ell_\mu$, while for $w$ on the rest of the boundary arcs we have $u(z) = |h_\alpha(z)| - 1$.
Thus, $u \le |h_\alpha|$ on the boundary of $U_\beta$ and hence also in $U_\beta$.
\end{proof}

The final result we need before proving Lemma~\ref{lem:rho_subharmonic} is the following inequality.

\begin{lem}\label{lem:h_shift_inequality}
We have $|h_\beta| \ge |h_\alpha \mp 1|$ whenever $\ell_\beta \subset \partial U_\alpha^\pm \smallsetminus \ell_\alpha$.
Moreover, the inequality is strict inside $U_\alpha^{\pm}$.
\end{lem}

\begin{proof}
Let us assume $\ell_\beta \subset \partial U_\alpha^+ \smallsetminus \ell_\alpha$; the case where  $\ell_\beta \subset \partial U_\alpha^- \smallsetminus \ell_\alpha$ is similar.
As both $|h_\beta|$ and $|h_\alpha - 1|$ are harmonic functions inside each tile $U_\mu^\pm$, it is enough to show that $|h_\beta(z)| \ge |h_\alpha(z) - 1|$ when $z \in \ell_\mu$ 
for every arc $\ell_\mu$.
Let us denote by $\mc T_\alpha(\nu)$ the subtree of $\mc T_\alpha$ rooted at $\nu$.
Case-by-case checking yields that
\[|h_\beta(z)| = \begin{cases}
 |h_\alpha(z) - 1|,& \text{when } \mu \in \{\alpha\} \cup \mc T_\alpha(\beta) \cup \bigcup_{\nu: \ell_\nu \subset \partial U_\alpha^- \smallsetminus \ell_\alpha} \mc T_\alpha(\nu) \\
 |h_\alpha(z) - 1| + 1, & \text{when } \mu \in \bigcup_{\nu: \ell_\nu \subset \partial U_\alpha^+ \smallsetminus (\ell_\alpha \cup \ell_\beta)} \mc T_\alpha(\nu)
\end{cases}.\qedhere\]
\end{proof}

We are now ready to prove Lemma~\ref{lem:rho_subharmonic}.

\begin{proof}[Proof of Lemma~\ref{lem:rho_subharmonic}]
It suffices to show that $\widetilde \rho$ is subharmonic in $\m D$. We will again use Lemma~\ref{lem:subharmonicity_criterion}. It is enough to consider the case $z \in \widehat \ell_\alpha$. We claim that one can choose $u = |h_\alpha - \widehat h_\alpha|$ in the domain $\widehat U_\alpha$, where $\widehat h_\alpha$ is constructed analogously to $h_\alpha$ but using the arcs $\widehat \ell_\beta$ instead. It is clear that $u(z) = \widetilde \rho(z)$ and from Lemma~\ref{lem:h_shift_inequality} we see that on $\widehat \ell_\beta \subset \partial \widehat U_\alpha^\pm \smallsetminus \widehat \ell_\alpha$ we have $\widetilde \rho = |h_\beta| \ge |h_\alpha \mp 1| = |h_\alpha - \widehat h_\alpha|$. Thus, it remains to show that $|h_\alpha - \widehat h_\alpha|$ is subharmonic near every $w \in \widehat U_\alpha$.

Notice that if $w \in \ell_\alpha$ or $w \in U_\beta$ for some $\beta \in I$, then $h_\alpha - \widehat h_\alpha$ is harmonic in a neighborhood of $w$ and thus $|h_\alpha - \widehat h_\alpha|$ is subharmonic in that same neighborhood. On the other hand if $w \in \ell_\beta$ for some $\beta \neq \alpha$, then $|h_\alpha(w)| \ge 1$ while $|\widehat h_\alpha(w)| < 1$. Thus $u = |h_\alpha| - \sgn(h_\alpha(w)) \widehat h_\alpha$ in a neighborhood of $w$. As $|h_\alpha|$ is subharmonic by Lemma~\ref{lem:halpha_subharmonic} and $\widehat h_\alpha$ is harmonic around $w$, we see that $u$ is subharmonic around $w$.
\end{proof}

Let us close this section by proving Theorem~\ref{thm:uniqueness}.

\begin{proof}[Proof of Theorem~\ref{thm:uniqueness}]
By Lemmas \ref{lem:rho_bounded} and \ref{lem:rho_subharmonic}, $\rho$ is a bounded subharmonic function on $\Sigma$ and hence equal to a constant $c \ge 0$.
Fix $\alpha \in I$. As we saw in the proof of Lemma~\ref{lem:rho_subharmonic}, we have $|h_\alpha - \widehat h_\alpha| \le \widetilde \rho = c$ on the boundary of $\widehat U_\alpha$.
Since $|h_\alpha - \widehat h_\alpha|$ is subharmonic in $\widehat U_\alpha$ and equals $c$ on $\widehat \ell_\alpha$, by maximum principle it has to equal $c$ everywhere in $\widehat U_\alpha$.
Now, if there is an arc $\widehat \ell_\beta \subset \partial \widehat U_\alpha$ that also intersects $U_\alpha^\pm$ at some point $z$, we get a contradiction due to the strict inequality $c = |h_\beta(z)| > |h_\alpha - \widehat h_\alpha|$ from Lemma~\ref{lem:h_shift_inequality}.
As the endpoints of $\ell_\alpha$ and $\widehat \ell_\alpha$ match, we have $U_\alpha \subset \widehat U_\alpha$, and since both $U_\alpha$ and $\widehat U_\alpha$ are fundamental domains for the covering map $\cover$, we must have $U_\alpha = \widehat U_\alpha$. In particular $\ell_\alpha = \widehat \ell_\alpha$.
\end{proof}

\section{Welding parametrization is a $C^1$ diffeomorphism} \label{sec:C1_lambda}

Recall that by Theorem~\ref{thm:piecewise_mob}, there is a correspondence between $C^1$ smooth piecewise geodesic loops and orientation-preserving piecewise Möbius $C^1$ diffeomorphisms of $\widehat{\mathbb{R}}$. 
The main aim of this section is to prove that this correspondence gives us a diffeomorphism between their parameter spaces.
We begin with a description of these spaces.
\begin{definition}
    For $n \ge 3$ we let $\mathcal{W}_n$ denote the space of   piecewise Möbius $C^1$ diffeomorphisms on $\widehat{\R}$ with $n$ marked points, %
    up to pre- and post-composition with global Möbius transformations. Formally, $\mathcal{W}_n$ can be defined as the set of equivalence classes of tuples $(\varphi; x_1, \dots, x_n)$ where $x_1,\dots,x_n$ is a set of distinct points ordered increasingly
    on $\widehat{\R}$ and $\varphi \colon \widehat{\R} \to \widehat{\R}$ is a $C^1$ diffeomorphism such that its restriction to every interval $[x_j,x_{j+1}]$ (with $x_{n+1} = x_1$) is a Möbius transformation in $\PSL(2,\R).$ Two tuples $(\varphi^\sigma; x_1^{\sigma}, \dots, x_n^{\sigma})$ ($\sigma = 1, 2$) are considered equivalent if there exist two global Möbius automorphisms $\alpha,\beta\in \PSL(2,\R)$ such that $\varphi^2 = \alpha \circ \varphi^1 \circ \beta$ and $\beta(x_j^2) = x_j^1$ for all $j=1,\dots,n$.
\end{definition}

There are several natural charts on $\mathcal{W}_n$ which makes it a manifold, such as that described in \cite[Thm.\,4.7]{MRW1}.
We will next describe a global chart 
which is particularly useful for our purpose. Notice first that every $p \in \mathcal{W}_n$ has a representative of the form $(\varphi; x_1, \dots, x_{n-3}, 0, 1, \infty)$ where $-\infty < x_1 < \dots < x_{n-3} < 0$ and $\varphi$ fixes $0,1,\infty$. From now on we will write $(\varphi_p; x_{p,1}, \dots, x_{p,{n-3}}, x_{p,n-2} = 0, x_{p, n-1} = 1, x_{p,n} = x_{p,0} = \infty)$ for this particular representative. We will also often drop the subindex $p$ for brevity.

As $\varphi_p$ is $C^1$ and %
analytic
outside of $\{x_j\}$, its pre-Schwarzian $\varphi_p''/\varphi_p'$ is a piecewise smooth function, and we write $\lambda_k$ for (the size of) the jump  at $x_k$ as in \eqref{eq:jump_Def}. 
For later use, we note that the Schwarzian derivative of $\varphi_p$ can be written as 
\[\mc S[\varphi_p] = \frac{\dd}{\dd x}\left(\frac{\varphi_p''}{\varphi_p'}\right) - \frac{1}{2} \left(\frac{\varphi_p''}{\varphi_p'}\right)^2 = \sum_{k=1}^{n-1} \lambda_k \delta_{x_k}\] by interpreting $\dd/\dd x$ as a distributional derivative,
where $\delta_x$ is the Dirac-measure at $x.$ Thus to every $p \in \mathcal{W}_n$ we can associate points
$x_1,...,x_{n-3}$ and jumps $\lambda_1,\dots,\lambda_{n-3},$ where the jumps satisfy certain constraints derived below. We will next reverse this process and use the variables $\lambda_k, 1\leq k\leq n-3,$ together with $x_k$ to define a global chart for $\mc W_n$.

Suppose that we are given $\lambda_1,\dots,\lambda_{n-3}$ and the points $x_1,\dots,x_{n-3}$. We will construct $\varphi$ by first constructing a map $\widetilde{\varphi}$ whose restriction to $[x_k,x_{k+1}]$, $k=0,\dots,n-3$ is given by $\widetilde{\varphi}_k = T_0 \circ T_1 \circ \dots \circ T_k$ where $T_0(x) = x$ and for $k\geq1,$
\begin{equation}\label{eq:Tk}T_k(x) \coloneqq \frac{\big(1 - \frac{1}{2}\lambda_k x_k)x + \frac{1}{2}\lambda_k x_k^2}{-\frac{1}{2} \lambda_k x + 1 + \frac{1}{2} \lambda_k x_k}
\end{equation}
is the Möbius map which satisfies $T_k(x_k) = x_k$, $T_k'(x_k) = 1$ and $T_k''(x_k) = \lambda_k$. Note that for any Möbius transform $\psi$ (particularly also for
$\psi = \widetilde{\varphi}_{k-1}$),
we have
\[\frac{(\psi \circ T_k)''(x_k)}{(\psi \circ T_k)'(x_k)} = \frac{\psi''(x_k)}{\psi'(x_k)} + \lambda_k,\]
so that the pre-Schwarzian of $\widetilde{\varphi}$ has the desired jumps $\lambda_k$ at $x_k,$ for $1\leq k\leq n-3.$

To define $\widetilde{\varphi}$ on 
$[0,\infty)=[x_{n-2}, x_{n-1}]\cup[x_{n-1}, x_{n}),$
assuming that $\widetilde{\varphi}_{n-3}(x_{n-2}) = \widetilde{\varphi}_{n-3}(0)\neq\infty,$ we let
$\widetilde{\varphi}_{n-2}$ be the (unique) Möbius map that agrees with 
$\widetilde{\varphi}_{n-3}$ at 0, has the same derivative there, has derivative $1$ at $1$, and does not have a pole on $[0,1]$ (there are two Möbius maps that satisfy these endpoint conditions, and only one of them satisfies $\widetilde{\varphi}_{n-2}(0)<\widetilde{\varphi}_{n-2}(1)$).
 Then we define $\widetilde{\varphi}_{n-1}$ to be the affine map of slope 1 that agrees with $\widetilde{\varphi}_{n-2}$ at 1.
Finally, the map $\varphi$ is obtained by composing with an affine transformation, \[\varphi(x) \coloneqq \frac{\widetilde{\varphi}(x) - \widetilde{\varphi}(0)}{\widetilde{\varphi}(1) - \widetilde{\varphi}(0)}.\]
This process always yields a piecewise-Möbius map, which is $C^1$ smooth at the $n$ marked points, but 
it is a homeomorphism only if none of the $\varphi_k$ has a pole on $[x_k,x_{k+1}]$ for $1\leq k\leq n-3$. 
The following proposition explicitly identifies the necessary and sufficient condition on $\lambda_k$.

\begin{prop} Given the parameters  $-\infty < x_1 < \dots < x_{n-3} < 0$ and $\lambda_1,\dots,\lambda_{n-3}\in\R$,
the above procedure yields a well-defined $C^1$ piecewise Möbius  homeomorphism of $\widehat{\R}$ if and only if 
the $\lambda_k$ satisfy the inequalities
\begin{equation}\label{eq:lambdaconstraint}-\infty < \lambda_k < \frac{2}{x_{k+1} - x_k} - \frac{\widetilde{\varphi}_{k-1}''(x_k)}{\widetilde{\varphi}_{k-1}'(x_k)}
\end{equation}
for every $k=1,\dots,n-3$. 
\end{prop}

\begin{proof}
    A short computation %
    shows that if $\psi$ is a Möbius map whose pre-Schwarzian at $x$ is $c$, then it has a pole at $x + 2/c$. Applying this to $\widetilde{\varphi}_k$ with $c = \widetilde{\varphi}_{k-1}''(x_k)/\widetilde{\varphi}_{k-1}'(x_k) + \lambda_k$, one easily checks that $x_k + 2/c \notin [x_k, x_{k+1}]$ 
    is equivalent to 
    \eqref{eq:lambdaconstraint}. Thus assuming \eqref{eq:lambdaconstraint} for all $k=1,\dots n-3,$ $\varphi$ is a homeomorphism between $[-\infty,0]$ 
    and $[-\infty,\widetilde{\varphi}_{n-3}(0)]$, the assumption $\widetilde{\varphi}_{n-3}(0)\neq\infty$ is satisfied, and $\widetilde{\varphi}$ (hence $\varphi$) is a global homeomorphism of $\widehat{\R}$.
\end{proof}

In particular, we see that $x_k$ and $\lambda_k$ with the above restrictions can be used to define a global chart and give $\mc W_n$ the structure of a smooth $2n-6$ dimensional manifold.

Let us now introduce for each $p \in \mc W_n$ the conformal maps $f_p \colon \H \to \Chat$ and $g_p \colon \H^* \to \Chat$ that solve the welding problem and satisfy $\varphi_p = g_p^{-1} \circ f_p$ while fixing $0,1,\infty$. The maps $f_p$ and $g_p$ map $\H$ and $\H^*$ to $\Omega_p^+$ and $\Omega_p^-$, respectively, and $f_p(\R) = g_p(\R)$ is the piecewise geodesic loop $\gamma_p$.

\begin{prop}\label{prop:f_differentiability}
    The map $\mc W_n \times \overline{\H} \ni (p,z) \mapsto f_p(z) \in \C$ is $C^1$ smooth and $f_p'(x) \neq 0$ for every $x \in \R$.
\end{prop}

\begin{proof}
    Note that by Lemma~\ref{lem:loop_geodesic_minimizer} %
    a $C^1$ piecewise geodesic Jordan curve is in fact $C^{1,1 - \varepsilon}$-smooth and hence by \cite[Thm.~3.5]{Pommerenke_boundary} we have $f_p'(x) \neq 0$ for every $x \in \R$.
    Since $p \mapsto \varphi_p$ is a differentiable map to $\mc W_n \to C^{1 + \varepsilon}(\R, \R)$, it follows, e.g., from \cite[Thm.~5.1]{de2003differentiability} that $p \mapsto f_p(z)$ is $C^1$ smooth.
\end{proof}

The main result of this section is that conformal welding yields a $C^1$ diffeomorphism $\mc W_n \to T_{0,n}$. More precisely, we have
\begin{thm}\label{thm:differentiability}
    The map $\mc F \colon \mc W_n \to \{\mc L(z_1,...,z_{n-3},0,1,\infty; \gamma )\} \cong T_{0,n}$ given by
    \[\mc F(p) = \Big(\big(f_p(x_{p,k})\big)_{k=1}^{n-3}, 0, 1, \infty; f_p(\R)\Big)\]
    is a $C^1$ diffeomorphism.
\end{thm}

By the existence and uniqueness of geodesic loops and their welding characterization, we already know that $\mc F$ is a bijection.
The main task left is to show that the inverse of $\mc F$ is also $C^1$ smooth. This will follow from the inverse function theorem once we show that the map $F \colon \mc W_n \to \Chat^{n-3}$, 
$$F(p) = \big(f_p(x_{p,k})\big)_{k=1}^{n-3}$$ %
has an invertible tangent map at every point. Or equivalently, that for any $v \in T_p \mc W_n$, we have $\partial_v F = 0$ if and only if $v = 0$. 
To do this, it is important to quantify how $f_p$ behaves on the real line, especially near the points $x_k$.

\begin{prop}\label{prop:f_asymptotics}
    For every $v \in T_p \mc W$ and $1 \le k \le n-3$ we have the asymptotics
    \[\frac{\partial_v f(z)}{f'(z)} = \frac{\partial_v f(x_k)}{f'(x_k)} + \Big(\frac{\partial_v f(x_k)}{f'(x_k)} + \partial_v x_k\Big) \frac{\lambda_k}{2\pi \ii} (z - x_k) \log(z - x_k) + O_{z \to x_k}(z-x_k),\]
    as well as
    \[\frac{\partial_v f(z)}{f'(z)} = O_{z \to 0}(z), \quad \frac{\partial_v f(z)}{f'(z)} = O_{z \to 1}(z - 1) \text{ and} \quad \frac{\partial_v f(z)}{f'(z)} = O_{z \to \infty}(z).\]
    Analogous claims hold for the map $g_p$.
\end{prop}

The proof of Proposition~\ref{prop:f_asymptotics} will be based on the following factorization result.

\begin{lem}\label{lem:factorization}
    For every $1 \le k \le n-1$ and $p \in \mc W_n$ there exist a unique conformal map $\psi\colon \m D \to (\Chat \smallsetminus \gamma) \cup \gamma_k \cup \gamma_{k-1} \cup \{z_k\}$, an angle $\theta \in (0,\pi)$ and a Möbius transformation $\alpha \in \PSL(2,\R)$ such that
    \[f = \psi \circ \mathfrak f_\theta \circ \alpha\]
    where $\mathfrak f_\theta$ is the conformal %
    map of a geodesic pair defined in Section~\ref{sub:pair-welding}.
    All of $\psi$, $\theta$ and $\alpha$ are $C^1$ functions of $p$.
\end{lem}

\begin{proof}
    Let $\psi \colon \m D \to (\Chat \smallsetminus \gamma) \cup \gamma_k \cup \gamma_{k-1} \cup \{z_k\}$ be the conformal map which takes $0$ to $z_k$, $\ee^{\ii\theta}$ to $z_{k-1}$ and $\ee^{-\ii\theta}$ to $z_{k+1}$, where $\theta \in (0,\pi)$ is determined by $\gamma$ and hence is a function of $p \in \mc W_n$. We may then write $f = \psi \circ \mathfrak f_\theta \circ \alpha$, where $\alpha$ is the Möbius transformation taking $x_k$ to $0$, $x_{k-1}$ to $1/D$ and $x_{k+1}$ to $1/B$, with $B = \big(\pi/2 - \theta\big) \cos(\theta) + \sin(\theta)$ and $D = \big(\pi/2  + \theta\big) \cos(\theta) - \sin(\theta)$ (we recall Figure~\ref{fig:G_maps} in Section~\ref{sub:pair-welding} for the mapping properties of $\mathfrak f_\theta$).

    Similarly, if we let $y_\ell = \varphi_p(x_\ell)$ for $\ell = 1,\ldots,n$ and let $\beta$ be the Möbius transformation taking $y_{k-1}$ to $1/C$, $y_k$ to $0$ and $y_{k+1}$ to $1/B$ with $C = -(3\pi/2) \cos(\theta) - \sin(\theta) + \theta \cos(\theta)$, then $g = \psi \circ \mathfrak g_\theta \circ \beta$, and in particular $\varphi_p(z) = \beta^{-1} \circ \varphi_\theta \circ \alpha(z)$ for $z \in [x_{k-1}, x_{k+1}]$, where $\varphi_\theta$ is the map from Lemma~\ref{lem:pair-welding-homeo}.
    If we compute the jump in the pre-Schwarzian at $z = x_k$ on both sides of this equation, we get by Lemma~\ref{lem:pair-welding-homeo} that $\lambda_k = -4 \pi \cos(\theta) \alpha'(x_k)$.
    Moreover, one can compute that
    \[\alpha'(x_k) = \frac{x_{k+1} - x_{k-1}}{(x_k - x_{k-1})(x_{k+1} - x_k) (B - D)},\]
    and noting that $B - D = 2 \sin(\theta) - 2 \theta \cos(\theta)$, we get
    \[\frac{2 \pi}{\theta - \tan(\theta)} = \frac{(x_k - x_{k-1})(x_{k+1} - x_k) \lambda_k}{x_{k+1} - x_{k-1}}.\]
    The left hand side is $C^1$ and strictly increasing function for $\theta \in (0,\pi)$ (when defined to be $0$ at $\theta = \pi/2$). In particular, by the inverse function theorem, $\theta$ is a $C^1$ function of $p \in \mc W_n$.
    It is easy to check that also $\mathfrak f_\theta$, $\alpha$ and $\alpha^{-1}$ are all $C^1$ functions of $p$.
    Moreover $\mathfrak f_\theta^{-1}$ and $\mathfrak g_\theta^{-1}$ are $C^1$ functions of $p$ and $z$ in their respective domains of definition (see Figure~\ref{fig:G_maps}).
    In particular we have either $\psi = f \circ \alpha^{-1} \circ \mathfrak f_\theta^{-1}$ or $\psi = f \circ \alpha^{-1} \circ \mathfrak g_\theta^{-1}$ depending on which connected component of the complement of the geodesic pair one looks at, and both formulas extend as $C^1$ smooth maps of $z$ to the boundary geodesic pair. In the closure of either connected component, $\psi$ is a $C^1$ smooth function of $p$. 
    Since $\psi$ is analytic in $z$, it follows from Morera's theorem that for any $v \in T_p \mc W_n$ the map $\partial_v \psi$ is analytic in $\m D$.
    In particular, $p \mapsto \psi$ is $C^1$ smooth in the whole unit disc.
\end{proof}

\begin{proof}[Proof of Proposition~\ref{prop:f_asymptotics}]
    Using the factorization, we now have
    \[\partial_v f = \partial_v \psi \circ \mathfrak f_\theta \circ \alpha + \psi' \circ \mathfrak f_\theta \circ \alpha \cdot \partial_v \mathfrak f_\theta \circ \alpha + \psi' \circ \mathfrak f_\theta \circ \alpha \cdot \mathfrak f_\theta' \circ \alpha \cdot \partial_v \alpha.\]
    Since $\partial_v \psi$ is $C^1$ smooth (or even analytic) in $\m D$, we have by Lemma~\ref{lem:ftheta-asymptotics} the asymptotics
    \[\partial_v \psi(\mathfrak f_\theta(\alpha(z))) = \partial_v \psi(0) + O_{z \to x_k}(z - x_k).\]
    By the same lemma we also have
    \[\partial_v \mathfrak f_\theta(z) = \partial_v \theta \cdot \Big(\frac{1}{2} \sin(\theta) z^2 \log(z) + O_{z \to 0}(z^2)\Big),\]
    which implies that $\psi'(\mathfrak f_\theta(\alpha(z))) \cdot \partial_v \mathfrak f_\theta(\alpha(z)) = o_{z \to x_k}(z - x_k)$.
    Finally (still by Lemma~\ref{lem:ftheta-asymptotics}) we have
    \[\mathfrak f'_\theta(z) = \frac{1}{2 \ii} - \cos(\theta) z \log(z) + O_{z \to 0}(z),\]
    giving us
    \begin{align*}
        & \psi'(\mathfrak f_\theta(\alpha(z))) \mathfrak f_\theta'(\alpha(z)) \partial_v \alpha(z) \\
        & = \frac{1}{2 \ii} \psi'(0) \partial_v \alpha(x_k) - \psi'(0) \cos(\theta) \partial_v \alpha(x_k) \alpha'(x_k) (z - x_k) \log(z - x_k) + O_{z \to x_k}(z - x_k).
    \end{align*}
    Similarly
    \begin{align*}
        & f'(z) = \psi'(\mathfrak f_\theta(\alpha(z))) \, \mathfrak f_\theta'(\alpha(z)) \, \alpha'(z) \\
        & = \frac{1}{2 \ii} \psi'(0) \alpha'(x_k) - \psi'(0) \cos(\theta) \alpha'(x_k)^2 (z - x_k) \log(z - x_k) + O_{z \to x_k}(z - x_k).
    \end{align*}
    Note that by differentiating the identity $\alpha(x_k) = 0$ we get $\partial_v \alpha(x_k) = -\alpha'(x_k) \partial_v x_k$
    and that Lemma~\ref{lem:pair-welding-homeo} together with the fact that $\varphi = \alpha^{-1} \circ \varphi_p \circ \alpha$ gives us $\lambda_k = \lambda[\varphi_\theta](0) \alpha'(x_k) = -4\pi \cos(\theta) \alpha'(x_k)$.
    Hence we get
    \[\partial_v f(z) = \partial_v f(x_k) + \frac{\lambda_k}{2\pi \ii} f'(x_k) \partial_v x_k (z - x_k) \log(z - x_k) + O_{z \to x_k}(z - x_k)\]
    and
    \[f'(z) = f'(x_k) - \frac{\lambda_k}{2\pi \ii} f'(x_k) (z - x_k) \log(z - x_k) + O_{z \to x_k}(z - x_k),\]
    which together yield the claimed asymptotics at $x_k$. The same computation shows the asymptotics at $0$ and $1$, since $\partial_v \alpha(0) = \partial_v \alpha(1) = 0$.
    Finally, to check the behavior at $\infty$, one can compare it to the behavior at $0$ after composing with $1/z$.
\end{proof}

We will next characterize those $v \in T\mc W_n$ for which $\partial_v \varphi = 0$. On the punctured sphere, this corresponds to the situation where the curve $\gamma$ stays fixed. The upshot is that this can only happen if some of $z_k$ move along the hyperbolic geodesic between $z_{k-1}$ and $z_{k+1}$.

\begin{lem}\label{lem:phi_p_null_perturbation}
    Suppose that $\partial_v \varphi = 0$ for some $v \in T\mc W_n$. Then $v = \sum_{k=1}^{n-3} c_k \mathbf{1}[\lambda_k = 0] \frac{\partial}{\partial x_k}$ for some constants $c_k \in \R$.
\end{lem}

\begin{proof}
    Let us consider both $\varphi$ and $\mc S[\varphi]$ as distributions on $\R^{2n-6+1}$. If $\partial_v \varphi = 0$, then since distributional derivatives commute we also have
    \[0 = \partial_v \mc S[\varphi] = \sum_{k=1}^{n-3} \partial_v[\lambda_k \delta(x - x_k)].\]
    Assuming that $v = \sum_{k=1}^{n-3} \big(c_k \partial/\partial x_k + d_k \partial/\partial \lambda_k \big)$ we have $$\partial_v[\lambda_k \delta(x - x_k)] = d_k \delta(x - x_k) - c_k \lambda_k \delta'(x - x_k),$$ which can only be $0$ if $d_k = 0$ and either $c_k = 0$ or $\lambda_k = 0$. 
\end{proof}

As a corollary, we get the following.

\begin{lem}\label{lem:trivial_v_moves_corners}
Suppose that $v \in T_p \mc W_n \smallsetminus \{0\}$ satisfies $\partial_v \varphi_p = 0$. Then $\partial_v F \neq 0$.
\end{lem}

\begin{proof}
As $\partial_v \varphi_p = 0$, by the aforementioned differentiability of the map $\varphi\mapsto f$ (\cite[Thm.~5.1]{de2003differentiability})
we also have $\partial_v f_p(x) = 0$. Hence $$\partial_v[f_p(x_{p,k})] = f_p'(x_{p,k}) \partial_v x_{p,k}$$ for every $1 \le k \le n-3$. The claim follows since by Proposition~\ref{prop:f_differentiability} we have $f_p'(x_{p,k}) \neq 0$ and by Lemma~\ref{lem:phi_p_null_perturbation} there exists $k \in \{1,\dots,n-3\}$ such that $\partial_v x_{p,k} \neq 0$.
\end{proof}

Geometrically, the next lemma says that if $\varphi_p$ is changing, then the loop $\gamma_p$ parametrized by $f_p$ is moving in its normal direction at some point on its boundary when $p$ moves along $v$.

\begin{lem}\label{lem:v_to_normal_variation}
   If $\partial_v \varphi_p \neq 0$, then for some $x \in \R$, we have $\Im\big(\partial_v f_p(x)/f_p'(x)\big) \neq 0$.
\end{lem}

\begin{proof}
    Let us set $h(z) = \partial_v f(z)/f'(z)$  for $z \in \overline{\H}$.
    By Proposition~\ref{prop:f_asymptotics} the analytic function $h(z)/(z(1-z))$ is bounded and $O(1/z)$ as $z \to \infty$. Assuming to the contrary that its imaginary part vanishes, we see that $h(z)/(z(1-z))$ has to be identically $0$, implying also that $\partial_v f(z) = 0$. A similar computation yields that $\partial_v g(z) = 0$ as well, and hence $\partial_v \varphi = \partial_v [g^{-1}] \circ f + [g^{-1}]' \circ f \cdot \partial_v f = 0$.
\end{proof}

The main part of our argument is the following lemma, whose proof uses ideas similar to those employed in Section~\ref{sec:uniqueness}.

\begin{lem}\label{lem:corners_need_to_move}
Suppose that $\partial_v F = 0$. Then $\Im\big(\partial_v f/f'\big) = 0$.
\end{lem}

\begin{proof}
Let $\varphi_k$ denote the $k-$th Möbius transformation of $\varphi$, i.e. $\varphi_k$ is defined on all of $\Chat$ and $\varphi|_{[x_k,x_{k+1}]}(x) = \varphi_k(x)$. We begin by defining for every $k \in \{0,\dots,n-1\}$ the function
\[f_k(z) = \begin{cases} f(z), \quad \text{if } z \in \H \cup \R \\ g \circ \varphi_k(z), \quad \text{if } z \in \H^*\end{cases}\]
which is a conformal map of the domain $\H \cup \H^* \cup (x_k,x_{k+1})$ onto the domain $\Omega_+ \cup \Omega_- \cup \gamma_k$.
We then define a harmonic function $u_k$ on $\Omega_+ \cup \Omega_- \cup \gamma_k$ by setting
\[u_k(f_k(z)) = \Im\Big(\frac{x_{k+1} - x_k}{(x_{k+1} - z)(z - x_k)} \cdot \Big(\frac{\partial_v f_k(z)}{f_k'(z)} + \frac{x_{k+1} - z}{x_{k+1} - x_k} \partial_v x_k + \frac{z - x_k}{x_{k+1} - x_k} \partial_v x_{k+1}\Big)\Big),\]
with the natural interpretations of the fractions if $x_k$ or $x_{k+1}$ equals $\infty$.
Note that we can also extend $u_k$ to the boundary $\gamma \smallsetminus \gamma_k$ of the domain as a bounded function. Indeed, on $\Omega^+$ the function is bounded since by the assumption $\partial_v F = 0$ we have $0 = \partial_v[f(x_k)] = \partial_v f(x_k) + f'(x_k)\partial_v x_k$. Together with Proposition~\ref{prop:f_asymptotics} this implies that
\[\frac{\partial_v f_k(z)}{f_k'(z)} + \frac{x_{k+1} - z}{x_{k+1} - x_k} \partial_v x_k + \frac{z - x_k}{x_{k+1} - x_k} \partial_v x_{k+1}\]
is $O(z - x_k)$ at $x_k$. Similarly one sees that it is also $O(z - x_{k+1})$ at $x_{k+1}$.
It also grows at most linearly so that together with the front factor, the quantity is bounded as $z \to \infty$. On $\Omega^-$, on the other hand, it suffices to consider what happens at the point $y_0 = \varphi_k^{-1}(\infty)$. If $y_0 = \infty$, then we are fine because $g(z) = O(z)$, $\varphi_k$ is linear, and the front factor will again make the quantity bounded. If $y_0$ is finite, then $f_k'(z) = g'(\varphi_k(z)) \varphi_k'(z)$ has a second order pole at $y_0$, while $\partial_v[g \circ \varphi_k](z) = \partial_v g(\varphi_k(z)) + g'(\varphi_k(z))\partial_v \varphi_k(z) = O((z-y_0)^{-2})$ is also blowing up at most quadratically. Note that when $z$ is real, we simply have
\[u_k(f_k(z)) = \Im\Big(\frac{x_{k+1} - x_k}{(x_{k+1} - z)(z - x_k)} \cdot \frac{\partial_v f_k(z)}{f_k'(z)}\Big).\]

Let us now define the function $u \colon \Chat \smallsetminus \{z_1,\dots,z_n\} \to \R$ by setting $u(z) = u_k(z)$ if $z \in \gamma_k$ and extending harmonically into the two domains $\Omega_+$ and $\Omega_-$. We claim that $|u|$ is a bounded subharmonic function and hence a constant. To see this, it is enough to check subharmonicity at $z \in \gamma_k$, which will follow if we can show that $|u_k(w)| \le |u(w)|$ for all $w \in \partial \Omega_k$. Let us first consider the case when $w = f(x) \in \Omega_k^+$ for some $x \in (x_\ell, x_{\ell+1})$ with $\ell \neq k$. In this case
\begin{align*}
    u_k(w) & = \Im\Big(\frac{x_{k+1} - x_k}{(x_{k+1} - x)(x - x_k)} \cdot \frac{\partial_v f(x)}{f'(x)}\Big) \\
    & = \frac{x_{k+1} - x_k}{(x_{k+1} - x)(x - x_k)} \cdot \frac{(x_{\ell+1} - x)(x - x_\ell)}{x_{\ell+1} - x_\ell} \cdot u(w).
\end{align*}
Let us assume that $x_{k+1} \le x_\ell$, the case $x_{\ell+1} \le x_k$ being similar. It is then enough to show that
\[\frac{x_{k+1} - x_k}{(x - x_{k+1})(x - x_k)} \le \frac{x_{\ell+1} - x_\ell}{(x_{\ell+1} - x)(x - x_\ell)},\]
but this is evident by writing both sides as partial fractions:
\[\frac{1}{x - x_{k+1}} - \frac{1}{x - x_k} \le \frac{1}{x_{\ell + 1} - x} + \frac{1}{x - x_\ell}.\]
This takes care of the $+$-side of the boundary of $\Omega_k$. For $\Omega_k^-$ let us write $y_\ell = \varphi(x_\ell)$ for all $\ell$ and suppose that $w = f_k(x) = g(\varphi_k(x)) \in \Omega_k^-$ for some $x \in (\varphi_k^{-1}(y_\ell), \varphi_k^{-1}(y_{\ell+1}))$.
Using that $f_k = f_\ell \circ \varphi_\ell^{-1} \circ \varphi_k$ on $\H^*$ we have
\begin{align*}
    f_k'(f_k^{-1}(w)) & = f_\ell'(f_\ell^{-1}(w)) \frac{\varphi_k'(f_k^{-1}(w))}{\varphi_\ell'(f_\ell^{-1}(w))}, \\
    \partial_v f_k(f_k^{-1}(w)) & = \partial_v f_\ell \circ (f_\ell^{-1}(w)) + f_\ell'(f_\ell^{-1}(w)) \partial_v [\varphi_\ell^{-1} \circ \varphi_k](f_k^{-1}(w))
\end{align*}
so that
\[\frac{\partial_v f_k(f_k^{-1}(w))}{f_k'(f_k^{-1}(w))} = \frac{\varphi_\ell'(f_\ell^{-1}(w))}{\varphi_k'(f_k^{-1}(w)} \left( \frac{\partial_v f_\ell(f_\ell^{-1}(w))}{f_\ell'(f_\ell^{-1}(w))} + \partial_v[\varphi_\ell^{-1} \circ \varphi_k](f_k^{-1}(w))\right).\]
Note that the second term is real when $w$ is on the boundary of $\Omega_k$. Hence, we have
\[u_k(w) = \frac{x_{k+1} - x_k}{(x_{k+1} - x)(x - x_k)} \cdot \frac{(x_{\ell + 1} - f_\ell^{-1}(w))(f_\ell^{-1}(w) - x_\ell)}{x_{\ell + 1} - x_\ell} \cdot \frac{\varphi_\ell'(f_\ell^{-1}(w))}{\varphi_k'(f_k^{-1}(w))} \cdot u(w).\]
Let us write $\psi_j(z) = \varphi_j^{-1}(z)$ for all $j \in \{1,\dots,n\}$ and also let $\varphi_k(x) = y$. Then we can write the above as
\[u_k(w) = \frac{(\psi_k(y_{k+1}) - \psi_k(y_k)) \psi_k'(y)}{(\psi_k(y_{k+1}) - \psi_k(y))(\psi_k(y) - \psi_k(y_k))} \cdot \frac{(\psi_\ell(y_{\ell + 1} - \psi_\ell(y))(\psi_\ell(y) - \psi_\ell(y_\ell))}{(\psi_{\ell}(y_{\ell + 1}) - \psi_\ell(y_\ell)) \psi_\ell'(y)} \cdot u(w).\]
As Möbius transformations preserve cross-ratios, one can (after taking a limit so that the derivative appears) deduce that the front factor equals
\[\frac{y_{k+1} - y_k}{(y_{k+1} - y)(y - y_k)} \cdot \frac{(y_{\ell + 1} - y)(y - y_\ell)}{y_{\ell + 1} - y_\ell},\]
which is at most $1$ in absolute value, as was seen above.

Thus, it follows that $|u|$ is subharmonic and hence a constant. In fact, as the inequalities obtained above are strict outside of the vertices, we must have $|u| = 0$ and thus $\Im(\partial_v f / f') = 0$.
\end{proof}

We are now ready to prove Theorem~\ref{thm:differentiability}.

\begin{proof}[Proof of Theorem~\ref{thm:differentiability}]
As indicated above, it suffices to show that for every $v \in T_p \mc W_n$ we have $\partial_v F = 0$ if and only if $v = 0$. By Lemma~\ref{lem:corners_need_to_move}, $\partial_v F = 0$ implies that $\Im\big(\partial_v f/f'\big) = 0$, which in turn by Lemma~\ref{lem:v_to_normal_variation} implies that $\partial_v \varphi = 0$. To avoid contradiction in Lemma~\ref{lem:trivial_v_moves_corners}, we then need to have $v = 0$.
\end{proof}

\section{%
Accessory parameters and Loewner energy} \label{sec:accessory}

The goal of this section is to show that the unique $C^1$ piecewise geodesic Jordan curve in the isotopy class $\mc L(z_1, \ldots, z_{n-3}, 0, 1, \infty; \tau)$ defines a special marked projective structure over the corresponding point in $T_{0,n}$ (Lemma~\ref{lem:geodesic_extension}). Furthermore, the Schwarzian derivative that compares this projective structure with the trivial projective structure has at most simple poles at the points $z_1, \ldots, z_{n-3}, 0, 1, \infty$, whose residues are given by the variation $\partial_{z_k} I^L(\gamma)/2$ of the minimal Loewner energy (Theorem~\ref{thm:Jordan_accessory}).
This result is analogous to the fact that the accessory parameters of the $n$-punctured sphere associated with the Fuchsian projective structure can be expressed using the classical Liouville action \cite{TZ1,TZ2}.

\subsection{Projective structure associated with a piecewise geodesic curve}
\label{sec:projective_from_curve}

We now show that any Jordan curve passing through $\{z_1,\cdots,z_n\}$ with the geodesic property determines a projective structure on the $n$-punctured sphere $\S = \Chat \smallsetminus \{z_1,\cdots,z_n\}$ by constructing a developing map. Note that in this subsection we do not assume that $\g$ is $C^1$ smooth.

As before, let $\g$ be a Jordan curve through points $z_1,\ldots,z_n$ with geodesic property in  $\mc L (z_1, \ldots, z_{n}; \tau)$,  
let $\O$ be the connected component of $\Chat \smallsetminus \g$ where $(z_1, \ldots, z_n)$ goes around $\O$ counterclockwise, and denote $\O^*$ the other connected component.
Let $f$ and $g$ be, respectively, a conformal map from $\m H$ to $\O$ and from $\m H^*$ onto $\O^*$. Recall that the welding homeomorphism of $\g$ is given by $\weld := g^{-1} \circ f|_{\widehat \R}$ %
and let $x_k=f^{-1}(z_k)$ and $y_k=g^{-1}(z_k)=\weld(x_k).$ 
We write $(x_k, x_{k+1})  = : l_k \subset \widehat \R$, $\varphi_k = \varphi|_{l_k} \in \PSL(2, \m R)$, and $\g_k$ the part of $\g$ between $z_k$ and $z_{k+1}$.

We now consider a conformal map $H: \O \cup \O^* \cup \g_1 \to \m H \cup \m H^* \cup l_1$ such that $H|_{\O} = f^{-1}$. This is possible since $\g_1$ and $l_1$ are, respectively, the conformal geodesic in $\O \cup \O^* \cup \g_1$ and $\m H \cup \m H^* \cup l_1$. In other words, the map $f^{-1}$ can be extended analytically across $\g_1$ via the map $H$. Similarly, $f^{-1}$ extends analytically across all $\g_k$, and each of these extensions can be analytically extended again across each $\g_j$.
\begin{lem}\label{lem:geodesic_extension}
By repeatedly extending $f^{-1}$ across every arc of $\g$ using the geodesic property, we obtain a holomorphic function $h: \widetilde \S \to \CC$, where $\widetilde \S$ is the universal cover of $\S$.
The function $h$ is a developing map of a projective structure on $\S$. Furthermore, the projective structure on $\S$ obtained does not depend on the choice of $f: \m H \to \O$.
\end{lem}
We denote the projective structure constructed this way by $Z_\g$.
\begin{proof}
We need to show that the projective structure on $\widetilde \S$ projects to a projective structure on $\S$, namely, $h \circ \cover^{-1}$ restricted to small open sets defines an atlas of local charts (projectively compatible while taking different lifts in $\cover^{-1}$), where $\cover: \widetilde \S \to \S$ is the covering map. 
Note that from the construction, for any simply connected pre-image $\widetilde \O = \cover^{-1} (\O)$,  we have $h|_{\widetilde \O}$ is a conformal map $\widetilde \O \to \m H$. 
Therefore, for all $\a \in \pi_1$, there exists $A_\a \in \PSL(2,\m R)$ such that  $h \circ \a = A_\a \circ h$. The local charts are, therefore, projectively compatible.

Any other choice of $f^{-1}$ is obtained by post-composing an element in $\PSL(2,\m R)$. Therefore they all give projectively compatible charts. 
\end{proof}

We can compute the holonomy representation of $Z_\g$ explicitly. 
\begin{prop}\label{prop:hol}
The holonomy representation of the projective structure with the developing map $h$ is given by 
$$\rho_\g(\eta_k) =  \weld_{k-1}^{-1} \circ \weld_k \in \PSL(2,\m R), \quad \text{ for all } k = 1, \ldots,n, $$ where $(\eta_1, \ldots, \eta_n)$ is the generator of $\pi_1$ described at the end of Section~\ref{sec:Teichmuller}.
Furthermore, $\g$ is $C^1$ smooth if and only if $\rho_\g(\eta_k)$ is  parabolic %
for all $1\le k\le n$.
\end{prop}

\begin{proof}
   We prove the statement for $\eta_2$, which starts from a point in $\O$, goes around $z_2$ by first crossing $\g_1$, then $\g_2$. The other cases are similar.
Since $g^{-1} \circ f|_{l_1} = \weld_1$ and $f^{-1}|_{l_1} = H|_{l_1}$, we have $g^{-1} = \weld_1 \circ H|_{\O^*}$ on $\g_1$. As both $g^{-1}$ and $H|_{\O^*}$ map $\O^*$ onto $\m H^*$, we have 
\begin{equation*}
\weld_1^{-1} \circ g^{-1} = H|_{\O^*},    
\end{equation*}
where with slight abuse of notation, $\weld_1$ refers to the M\"obius map in $\PSL(2,\m R)$ acting on $\Chat$. 

The geodesic property allows us to extend analytically $H|_{\O}$ across $\g_2$ to a conformal map $\widehat H :  \O \cup \O^* \cup \g_2 \to \m H \cup \m H^* \cup l_2$ such that $\widehat H|_{\O^*} = H|_{\O^*}$.  
We have 
$$\widehat H|_{\g_2} = H|_{\g_2} = \weld_1^{-1} \circ g^{-1} |_{\g_2} = \weld_1^{-1} \circ \weld_2 \circ f^{-1} |_{\g_2}.$$
As both $\widehat H$ and $f^{-1}$ map $\O$ onto $\m H$, we obtain
\begin{equation}\label{eq:holonomy}
\widehat H |_{\O} = \weld_1^{-1} \circ \weld_2 \circ f^{-1}.    
\end{equation}
In terms of the developing map, we obtain that $h \circ \eta_2 = \weld_1^{-1} \circ \weld_2 \circ h$. That is $\rho_\g(\eta_2) =  \weld_1^{-1} \circ \weld_2 \in \PSL(2,\m R)$.

Since $\weld$ is continuous, we have $\weld_1 (x_2) = \weld_2 (x_2)$ where  $x_k =f^{-1} (z_k)$. Therefore,
$\rho_\g(\eta_2) (x_2) = x_2$.
If $\g$ is $C^1$ smooth, from Theorem~\ref{thm:piecewise_mob}, we have $\weld'$ is continuous. Hence, $(\weld_1^{-1} \circ \weld_2)'(x_2) =1$ which shows that $\rho_\g(\eta_2)$ is parabolic.
\end{proof}

Summarizing the observation above, we obtain:
\begin{thm}\label{thm:Z_gamma}
A Jordan curve $\g \in  \mc L(z_1, \ldots, z_{n}; \tau)$ with the geodesic property determines a projective structure $Z_\g$ with a developing map $h$ whose holonomy representation takes values in $\PSL(2,\m R)$. 
 Moreover, if $\g$ is $C^1$ smooth, the holonomy around each puncture $z_k$ is parabolic. 
\end{thm}

\begin{remark}
Although the Fuchsian projective structure $Z_F$ (see Example~\ref{ex:projective structures}) on $\S$ also has holonomy representation in $\PSL(2,\m R)$ and $\rho_F(\eta_k)$ are parabolic, $Z_\g$ is different as $h$ maps onto $\CC$ and cannot be projectively equivalent to the uniformizing map $h_F$.
\end{remark}

The rest of Section~\ref{sec:accessory} studies the Schwarzian derivative comparing $Z_\g$ to $Z_F$. %

\subsection{The variation of the energy of geodesic pairs}
\label{sec:variation_pair}

The goal of this section is to show the following result. 
We fix a simply connected domain with two marked boundary points $(D;a,b)$. For $z \in D$, we define
$$I_{D;a,b} (z) =\min_{\g} I_{D;a,b} (\g) = I_{D;a,b} (\g^z) $$
where the minimum is taken over all chords  in $(D;a,b)$ passing through $z,$ and $\g^z$ is the unique energy minimizer as in Section~\ref{sub:geodesic}.

\begin{thm}\label{thm:C_I_pair}
     Let $h$ be a uniformizing map from a connected component of $D \smallsetminus \g^z$ to $\m D$. Then the Schwarzian derivative $\mc S [h]$ extends to an analytic function in $D\smallsetminus \{z\}$, has a simple pole at $z$, and
   \begin{equation}\label{eq:pair_residue}
   \mc S [h](w) \sim \frac{\partial_z I_{D;a,b} (z)/2}{w -z} \qquad \text{as } w \to z,
   \end{equation}
    where $\partial_z = (\partial_x - \ii \partial_y)/2$ is the Wirtinger derivative.
\end{thm}
 
The fact that $\mc S[h]$ extends to an analytic function in $D \smallsetminus \{z\}$ follows from the geodesic property of $\g$, see Lemma~\ref{lem:geodesic_extension}.  The proof of \eqref{eq:pair_residue} has two parts. We first use an explicit expression of $h$ to compute $\mc S[h]$ in the case of a geodesic pair $(\m D; \ee^{\ii \t}, \ee^{-\ii \t}; 0)$ (Lemma~\ref{lem_pole_disk}), then relate the residue to the Loewner energy and extend the result to general geodesic pairs.

For any conformal map $f$, if the Schwarzian $\mc S [f]$ has a simple pole at $z$, we write $C(f)(z) \in \m C$ for the residue of $\mc S[f]$ at $z$ so that $$\mc S[f] (w) \sim \frac{C(f)(z)}{w-z}, \quad \text{as } w \to z.$$ 

Note that $C(f)$ transforms like a $(1,0)$-form:
\begin{lem}\label{lem:C_1_0_form}
If $g$ is a conformal map and $g(z)$ is a simple pole of $\mc S[f]$, then $z$ is a simple pole of $\mc S[f \circ g]$ and
\begin{equation*}
C(f \circ g) (z) = C (f) \circ g(z) \cdot g'(z).
\end{equation*}
If $z$ is a simple pole of $\mc S[g]$ and $f \in \PSL(2,\C)$, then 
$$C(f \circ g) (z) = C(g) (z).$$
\end{lem}
\begin{proof}
   From the chain rule for Schwarzian derivatives,  $
       \mc S [f \circ g] (w)  = \mc S [f] (g(w)) g'(w)^2 + \mc S [g] (w)$. As $\mc S [g]$ is holomorphic at $z$,  we have
  $$\mc S [f \circ g] (w)  \sim \frac{C(f)(g(z))}{g(w)-g(z)} [g'(w)]^2 \sim \frac{C(f)(g(z)) g'(z)}{ w-z}$$
  as $w \to z$. Similarly, if $f \in \PSL(2,\C)$, then $\mc S[f \circ g] = \mc S[g]$ which shows the second claim.
\end{proof}

\begin{lem}\label{lem_pole_disk}
For $ 0 \le \t < \pi /2$, let $h_\t$ be the conformal map $\m D \smallsetminus \gamma \to \H \cup \H^*$ defined in Section~\ref{sub:pair-welding}.
   The Schwarzian of $h_\t$ extends meromorphically to $\m D$ and has a simple pole at $0$ with $C(h_\t) (0)=-4 \cos (\t)$.
\end{lem}
\begin{proof}
This is essentially Corollary 3.11 in \cite{MRW1}, with a slightly altered normalization. Indeed, Equation~(6) in \cite{MRW1} and Lemma~\ref{lem:C_1_0_form}  show that $\mc S[h_\t]$ has a simple pole at $0$ and
  $$C(h_\t) (0)  = \ii C(G_{\pi/2 -\t}) (0) = -4 \cos(\t)$$
  as claimed.
\end{proof}

We will now relate the pole of the Schwarzian to the minimal Loewner energy. For this, we take the domain $(\m H; 0, \infty)$ as a reference domain and vary the interior marked point. There is a unique M\"obius map $\phi_\t: \m D \to \m H$ sending $\ee^{\ii \t}$ to $0$, $\ee^{-\ii \t}$ to $\infty$, and $0$ to a point on the unit circle, namely
$$\phi_\t(w) = \frac{w-\ee^{\ii \t}}{w \,\ee^{\ii \t} - 1}.$$
In particular, $\phi_\t (0) = \ee^{\ii \t}$ and $\phi_\t'(0) = \ee^{2\ii \t} -1 = 2 \ii \ee^{\ii \t} \sin (\t)$.
\begin{figure}[H]
\centering
  \includegraphics[width=0.6\linewidth]{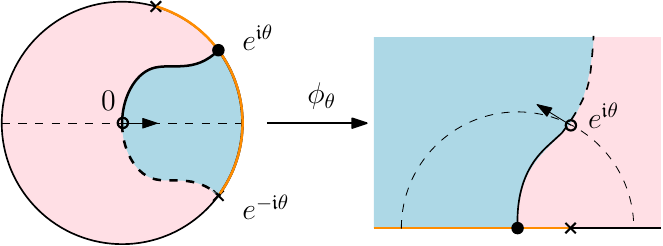}
  \caption{\label{fig_schwarzian_H} $C^1$ smooth geodesic pairs in $\m D$ and $\m H$.}
\end{figure}

\begin{lem}\label{lem:H_Sch_I} 
The Schwarzian derivative of the uniformizing map $h_\t \circ \phi_\t^{-1}$ extends meromorphically with a simple pole at $z = \ee^{\ii \t}$, and
$$C (h_\t \circ \phi_\t^{-1}) (\ee^{\ii\t}) = \partial_z I_{\m H; 0, \infty} (z)/2|_{z = \ee^{\ii \t}}.$$
\end{lem}
\begin{proof}
   It follows from Lemma~\ref{lem:C_1_0_form} and Lemma~\ref{lem_pole_disk} that 
   $$C (h_\t \circ \phi_\t^{-1}) (\ee^{\ii \t})= -\frac{4 \cos \t}{\phi'_\t (0)} = -\frac{2 \cot \t}{\ii \ee^{\ii \t}}.$$
   On the other hand, by Lemma~\ref{lem:8ln}, 
   $$I_{\m H; 0, \infty} (z) = -8 \log \frac{\Im z}{\abs z} = - 8 \log \frac{z - \ad z}{ 2\ii} + 4 \log z \ad z. $$
   We have
   $$\partial_z I_{\m H; 0, \infty} (z)|_{z= \ee^{\ii \t}} = \frac{-8}{z-\ad z} + \frac{4}{z} \Big\vert_{z = \ee^{\ii \t}} = \frac{-4 (z + \ad z)}{(z-\ad z)z}\Big\vert_{z = \ee^{\ii \t}} = \frac{-4 \cot \t}{\ii \ee^{\ii \t}} = 2 C (h_\t \circ \phi_\t^{-1}) (\ee^{\ii \t})$$
   as claimed.
\end{proof}
\begin{proof}[Proof of Theorem~\ref{thm:C_I_pair}]
The Loewner chordal energy is conformally invariant, i.e., for $\psi: \m H \to D$ a conformal map such that $a = \psi (0)$ and $b = \psi(\infty)$, we have
$I_{D;a,b} (\psi(\g)) = I_{\m H; 0, \infty} (\g).$
Then we clearly have
$$I_{D;a,b} (\psi(z)) = I_{\m H; 0, \infty} (z).$$
This implies that $\partial_z I_{D;a,b}(z)$ also transforms like a $(1,0)$-form:
$$\partial_z I_{\m H; 0,\infty} (z) = (\partial_{z} I_{D;a,b})\circ \psi (z)\cdot \psi'(z),$$
with concludes the proof by Lemma~\ref{lem:C_1_0_form} and Lemma~\ref{lem:H_Sch_I}.
\end{proof}

\subsection{Schwarzian and the jump parameter}

We now relate the residue in the Schwarzian to the parameter $\lambda$ of the jump of the pre-Schwarzian of the welding homeomorphism as defined in \eqref{eq:jump_Def}.

Let $\g$ be the $C^1$ smooth geodesic pair in $(D;a,b;z_0)$, $\O$ the connected component of $D \smallsetminus \g$ where $(a,z_0,b)$ are in counter-clockwise order in $\partial \O$, and $\O^*$ the connected component where $(a,z_0,b)$ are in clockwise order of $\partial \O^*$. Let $f : \m H \to \O$ and $g : \m H^* \to \O^*$ be conformal maps such that the $x_0 := f^{-1}(z_0)$ and $g^{-1}(z_0)$ are not $\infty$. As before, we write the part of $\g$ from $a$ to $z_0$ as $\g_1$, and the part from $z_0$ to $b$ as $\g_2$.
\begin{thm}\label{thm:lambda_C}
The welding homeomorphism of $\g$, defined as $\varphi := g^{-1} \circ f$  on the interval $f^{-1} (\g) \subset \widehat {\m R}$, is $C^1$, and in $\PSL(2,\m R)$ on $I_1 := f^{-1} (\g_1)$ and on $I_2 : = f^{-1} (\g_2)$. 
Moreover, the jump $\lambda [\varphi] (x_0)$ of %
$\varphi''/\varphi'$ at $\{x_0\} = I_1 \cap I_2 = \{f^{-1} (z_0)\}$ satisfies
    \begin{equation} \label{eq:lambda_C}
    \frac{\lambda [\varphi] (x_0)}{f'(x_0)} = 2 \pi \ii \,C[f^{-1}] (z_0).
    \end{equation}
\end{thm}

\begin{proof}
    We consider first the case of a geodesic pair in $(\m D ; \ee^{\ii \theta}, \ee^{-\ii \theta}; 0)$, where $\theta \in (0, \pi /2]$ and choose $f = \mathfrak f_\theta$ and $g = \mathfrak g_\theta$ as defined in Section~\ref{sub:pair-welding}.
    Recall that by Lemma~\ref{lem:pair-welding-homeo} we have
    \begin{equation}\label{eq:lambda_explicit_theta}
        \lambda[\varphi_\theta] (0) = -4 \pi \cos (\theta).
    \end{equation}
    and by Lemma~\ref{lem:ftheta-asymptotics} we also have
    \[\mathfrak f_\theta'(0) = (2 \ii)^{-1}.\]

On the other hand, $\mathfrak f_\theta^{-1}$ has the same Schwarzian as $h_\theta$, which was computed in Lemma~\ref{lem_pole_disk}. Hence by \eqref{eq:lambda_explicit_theta} we have
$$2 \pi \ii C[\mathfrak f_\theta^{-1}] (0)  = - 8 \pi \ii \cos (\theta) = \frac{\lambda[\varphi_\theta](0)}{\mathfrak f_\theta'(0)}.$$

    Now we consider an arbitrary $C^1$ smooth geodesic pair in $(D; a,b; z_0)$ and conformal maps $f : \m H \to \O$ and $g : \m H^* \to \O^*$. From the assumption,  $f^{-1} (z_0) = x_0 \in \m R$. 
    We assume that $D$ is a simply connected domain conformally equivalent to $\m D$ and $a \neq b$. Otherwise, we can cut open a piece of $\g_1$ near $a$, and the remaining curve is still a geodesic pair in the complement of the piece, and the corresponding welding is simply the restriction to a smaller interval containing $x_0$.
    
    There exists $\theta \in (0, \pi/2]$ such that either $(D; a,b; z_0)$ or $(D; b,a; z_0)$ is conformally equivalent to $(\m D; \ee^{\ii \theta}, \ee^{-\ii \theta}; 0)$.
Assume without loss of generality that $(D; a,b; z_0)$ is conformally equivalent to $(\m D; \ee^{\ii \theta}, \ee^{-\ii \theta}; 0)$ by a conformal map $\psi : \m D \to D$ respecting the marked points (otherwise, we switch $f$ and $g$ and replace $\varphi$ by $\varphi^{-1}$). 

Now let $\alpha \in \PSL(2,\m R)$ such that $f = \psi \circ \mathfrak f_\theta \circ \alpha$. In particular, $\alpha (x_0) = 0$. Similarly, let $\beta \in \PSL(2, \m R)$ such that $g = \psi \circ \mathfrak g_\t \circ \beta$. The assumption that $g^{-1} (z_0) = : y_0 \neq \infty$ implies that $\beta(y_0) = 0$ and $\beta^{-1}$ is regular near $0$. 
We obtain 
$$\varphi = \beta^{-1} \circ \varphi_\theta \circ \alpha.$$
Using the chain rules
$$\lambda [\varphi] (x_0) = \lambda [\varphi_\t] (0) \alpha'(x_0), \qquad f'(x_0) = \psi'(0) \mathfrak f_\t'(0) \alpha'(x_0)$$
and
$$C [f^{-1}] (z_0) = C [\mathfrak f_\t^{-1} \circ \psi^{-1}] (z_0) = \frac{C [\mathfrak f_\t^{-1}] (0)}{\psi'(0)},$$
we obtain 
$$\frac{\lambda [\varphi] (x_0) }{f'(x_0)} = \frac{ \lambda [\varphi_\t] (0)}{ \psi'(0) f_\t'(0)} = 2\pi \ii \frac{C [\mathfrak f_\t^{-1}] (0)}{\psi'(0)} = 2\pi \ii C [f^{-1}] (z_0)$$
which completes the proof.
\end{proof}

\subsection{Schwarzian and piecewise geodesic Jordan curve}

We fix an isotopy class $\mc L (z_1, \ldots, z_{n}; \tau)$.
   Let $\g$ be the unique piecewise geodesic Jordan curve in this class and define $I(z_1, \ldots, z_{n}; \tau) = I^L(\g)$. 
   As before, we write $\g \smallsetminus \{z_1,\dots,z_n\} = \g_1 \cup \ldots \cup \g_n$ where $\g_k$ connects $z_k$ to $z_{k+1}$, and $\O$ (resp., $\O^*$) for the connected component of $\CC \smallsetminus \g$ where the marked points are ordered counterclockwise (resp. clockwise) on the boundary.  
   For small enough $\vare>0$, the ball $B_{\vare} (z_k)$ of Euclidean radius $\vare$ centered at $z_k$ does not intersect $\g \smallsetminus (\g_{k-1} \cup \g_k \cup \{z_k\})$.
   For $z \in B_{\vare} (z_k)$,
   we write 
   $\g^z_k$ for the unique $C^1$ smooth geodesic pair in $(\O \cup \O^* \cup \g_{k-1} \cup \g_k \cup \{z_k\}; z_{k-1}, z_{k+1}; z)$.
  We define $I(z_1, \ldots, z_{k-1}, z, z_{k+1}, \ldots z_{n}; \tau)$ to be the minimal energy of Jordan curves isotopic to $(\g \smallsetminus \g^{z_k}_k) \cup \g^z_k$ (relative to $z_1, \ldots, z_{k-1}, z, z_{k+1}, \ldots z_{n}$). %
   
\begin{thm}\label{thm:Jordan_accessory}
Let $Z_\g$ be the associated projective structure defined in Section~\ref{sec:projective_from_curve} and $q$ the corresponding quadratic differential defined in Lemma~\ref{lem:quadratic_diff}. Assume that $z_k\neq \infty$ for all $k = 1, \ldots, n$.    
We have
$$q (w) = \sum_{k = 1}^n \frac{\partial_{k} I(z_1, \ldots, z_{n}; \tau)/2}{w - z_k} \,\dd w^2,$$
where $\partial_k = (\partial_{x_k} - \ii \partial_{y_k})/2$ if $z_k = x_k + \ii y_k$.
\end{thm}

\begin{remark} Let us first make a few remarks.
\begin{itemize}
    \item If we normalize the $n$-punctured sphere such that $z_{n-2} = 0$, $z_{n-1} = 1$, and $z_n = \infty$, then the expression of $q$ can be easily obtained using the property of quadratic differentials, see \eqref{eq:quadratic_differential}. The assumption of $\infty$ not being a marked point makes the expression of $q$ cleaner, so we only consider this case.

\item This result is reminiscent to the result on the accessory parameter of the Fuchsian projective structure, first conjectured by Polyakov and proved by Takhtajan and Zograf in \cite[Thm.\,1]{TZ1}, which shows that the associated quadratic differential is given by
$$q_F(w) = \sum_{k = 1}^n \left(\frac{1}{2(w -z_k)^2}  - \frac{\partial_k S_{cl}}{2\pi}\frac{1}{(w-z_k)}\right) \,\dd w^2$$
where $S_{cl}$ is the classical Liouville action on the $n$-punctured sphere.
\item Note that $q$ also uniquely determines $\g$. In fact, let $q$ and $\widehat q$ be the quadratic differentials associated with two piecewise geodesic Jordan curves $\g$ and $\widehat \g$. If $q = \widehat q$, then Lemma~\ref{lem:quadratic_diff} shows that $Z_{\g} = Z_{\widehat \g}$. In particular, there exists $A \in \PSL(2,\C)$ such that $\widehat h = A\circ h$, where $h$ and $\widehat h$ are the associated developing maps with the holonomy representation in $\PSL(2,\R)$. As any lift $h \circ \cover^{-1}|_{\O}$ maps $\O$ to $\m H$ and the marked points into $\widehat \R$, as $n \ge 3$, $A$ has to map at least $3$ points on $\widehat \R$ into $\widehat \R$, hence $A \in \PSL(2,\R)$. 
Since $\g = \cover \circ h^{-1} (\widehat \R)$ and $h^{-1} (\widehat \R) = \widehat h^{-1} \circ A^{-1} (\widehat \R) = \widehat h^{-1} (\widehat \R)$, we obtain $\g = \widehat \g$.
 \end{itemize}
\end{remark}

\bigskip
\begin{prop}\label{prop:same_der}
The function
$$I (\cdot; \tau) \colon (z_1, %
\ldots, z_{n}) \mapsto I(z_1, \ldots, %
z_{n}; \tau)$$ is $C^1$ smooth and for all $k = 1, \cdots, n$,
$$\partial_{z_k} I(z_1, \ldots, z_{n}; \tau) = \partial_z I_{D_k; z_{k-1}, z_{k+1}} (z_k),$$
where $D_k = \O \cup \O^* \cup \g^{z_k}_k$ and $I_{D_k; z_{k-1}, z_{k+1}}(z) = I_{D_k; z_{k-1}, z_{k+1}}(\g^z_k)$.
\end{prop}
\begin{proof}
We first notice that if $I(z_1, %
\ldots, z_{n}; \tau)$ is differentiable at $z_k$, then 
\begin{equation}\label{eq:der_eq}
    \partial_{z_k} I(z_1, \ldots, z_{n}; \tau) = \partial_z I_{D_k; z_{k-1}, z_{k+1}} (z_k),
\end{equation}
since by definition, 
$$I(z_1, \ldots, z_{k-1}, \cdot, z_{k+1}, \ldots, z_{n}; \tau) \le I_{D_k; z_{k-1}, z_{k+1}} (\cdot) + I^A (\g \smallsetminus \g^{z_k}_k)$$
and they coincide at $z_k$.

To show the differentiability, we show first that $ I(z_1, \ldots, z_{k-1}, \cdot, z_{k+1}, \ldots, z_{n}; \tau)$ is locally Lipschitz. From Theorem~\ref{thm:differentiability}, the unique piecewise geodesic Jordan curve $\g^{z_k}$ in $\mc L (z_1, \cdots, z_n; \tau)$ depends continuously on $z_k$ with respect to the Hausdorff metric on the sphere.  
Therefore for every $z_k$, there exist a ball $B_r(z_k)$ with some radius $r>0$ centered at $z_k$, such that for all $z \in B_r(z_k)$, we have $B_{2r}(z_k) \cap \widehat \g^{z} = \emptyset$ where $\widehat \g^{z}$ is the arc in $ \g^{z}$ obtained by removing the geodesic pair passing through $z$.  In particular, this implies that the image of $B_r (z_k)$ under the conformal map $h_z : \Chat \smallsetminus \widehat \g^z \to \m H$ sending $z_{k -1} \mapsto 0$, $z_{k+1} \mapsto \infty$, $z \mapsto \ee^{\ii \theta}$ for some $\theta$, is contained in a compact set $K \subset \m H$ (which is independent of $z \in B_r (z_k)$). 

 Therefore, for $z,z' \in B_r(z_k)$,
\begin{align*}
& I(z_1, \ldots, z_{k-1}, z', z_{k+1}, \ldots, z_{n}; \tau) - I(z_1, \ldots, z_{k-1}, z, z_{k+1}, \ldots, z_{n}; \tau) \\
& \le I^A (\widehat \g^z) + I_{\Chat \smallsetminus \widehat \g^z; z_{k-1}, z_{k+1}} (z') - I(z_1, \ldots, z_{k-1}, z, z_{k+1}, \ldots z_{n}; \tau)\\
& =  I_{\m H; 0, \infty} (h_{z} (z'))  - I_{\m H; 0, \infty} (h_{z} (z)).
\end{align*}
Since $h_{z} (z'), h_{z} (z) \in K$, $I_{\m H; 0, \infty} (\cdot)$ is smooth in $K$ and $h_{z}'$ is uniformly bounded in $B_r (z_k)$ for all $z \in B_r (z_k)$, we obtain that there exists $c = c (z_k, r)$ such that for all $z, z' \in B_r (z_k)$,
$$I(z_1, \ldots, z_{k-1}, z', z_{k+1}, \ldots, z_{n}; \tau) - I(z_1, \ldots, z_{k-1}, z, z_{k+1}, \ldots, z_{n}; \tau)  \le c|z' - z|.$$
This proves that $I (\cdot; \tau)$ is locally Lipschitz.

From the Rademacher theorem, the Lipschitz function 
$I (\cdot; \tau)$  is differentiable almost everywhere. 
Let $(z_1, \ldots, z_{k-1}, z_k, z_{k+1}, \ldots, z_{n})$ be a differentiable point of  $I (\cdot; \tau)$. 
Let $\g$ be the unique minimizer in the isotopy class $(z_1, \ldots, z_{k-1}, z_k, z_{k+1}, \ldots, z_{n}; \tau)$, and let $f: \m H \to \O$ and $g : \m H^* \to \O^*$ be two conformal maps as before. We parametrize the welding homeomorphism $\varphi = g^{-1} \circ f$ by the coordinates $(x_k, \lambda_k = \lambda [\varphi] (x_k))$ as in Section~\ref{sec:C1_lambda}.
We have for all $k = 1, \cdots,  n $,
$$\partial_{z_k} I(z_1, \ldots, z_{n}; \tau) = \partial_z I_{D_k; z_{k-1}, z_{k+1}} (z_k) = 2 C[f^{-1}] (z_k) = \frac{\lambda_k}{\pi \ii f'(x_k)}.$$
The first equality follows from the observation \eqref{eq:der_eq}, the second equality follows from Theorem~\ref{thm:C_I_pair}, and the last equality follows from Theorem~\ref{thm:lambda_C}. 

Recalling the factorization $f = \psi \circ \mathfrak f_\theta \circ \alpha$ from Lemma~\ref{lem:factorization} we see that $f'(x_k) = \frac{1}{2 \ii}\psi'(0) \alpha'(x_k)$, which is a continuous function of $p \in \mc W_n$. Hence $\lambda_k/f'(x_k)$ depends continuously on $(z_1, \cdots, z_n)$.
Thus the almost everywhere defined partial derivatives of $I(\cdot; \tau)$ extend to continuous functions.
As $I(\cdot;\tau)$ is absolutely continuous on lines, it is determined by its partial derivatives and hence is a $C^1$ function, which completes the proof.
\end{proof}

\begin{proof}[Proof of Theorem~\ref{thm:Jordan_accessory}]

Recall that the quadratic differential $q$ is defined as $\mc S [h \circ \cover^{-1}]$. Here, $h : \widetilde \S \to \CC$ is a developing map constructed in Section~\ref{sec:projective_from_curve} such that any lift of $h \circ \cover^{-1}$ restricted to $\O$, is a conformal map from $\O$ onto $\m H$. 
Since for all $1 \le k\le n$, $\g^{z_k}_k = \g_{k-1} \cup \g_k \cup \{z_k\}$ is a geodesic pair in $(D_k; z_{k-1}, z_{k+1})$, Theorem~\ref{thm:C_I_pair} and Proposition~\ref{prop:same_der} show that 
$$q (w) \sim \frac{\partial_z I_{D_k; z_{k-1}, z_{k+1}} (z_k)/2}{w - z_k} =  \frac{\partial_k I(z_1, \ldots z_{n}; \tau)/2}{w - z_k} , \qquad \text{as } w \to z_k.$$

As $q$ does not have a pole at $\infty$, as a Schwarzian derivative, it means that 
$$q (w) = O(w^{-4}) \qquad \text {as } w\to \infty.$$
We obtain the claimed expression of $q$ since the difference of the left- and right-hand sides is holomorphic on $\CC$ and vanishes as $w \to \infty$.
\end{proof}

\section{Relation to Fuchsian projective structures} \label{sec:grafting}

In this section, we will show how one can obtain piecewise geodesic Jordan curves and their associated projective structures by grafting the Fuchsian projective structures.%

We first notice that there is another natural way to associate a unique Jordan curve in an isotopy class $\mc L (w_1, \ldots, w_n; \tau)$ with a marked Fuchsian projective structure, which gives another parametrization of $T_{0,n}$ by Jordan curves.

\begin{lem}\label{lem:geo_puncture}
Let $w_1,\ldots,w_n \in \Chat$, $n\ge 3$,  be  distinct points, $\tau$ be a Jordan curve passing through $w_1, \ldots, w_n$ in cyclic order, and  $%
\Sigma' \coloneqq \Chat \smallsetminus \{w_1,\dots,w_n\}$ be the $n$-punctured sphere endowed with its  complete hyperbolic metric. Then 
in the isotopy class $\mc L (w_1, \ldots, w_n; \tau)$ there exists a unique Jordan curve $\gamma'$  such that each arc of $\gamma'$ between consecutive cusps  is a hyperbolic geodesic
in $\Sigma'$. 
\end{lem}
\begin{proof}
    We denote by $\cover' \colon \H \to \Sigma'$ the (essentially unique) holomorphic universal covering map and by $\G' \subset \PSL(2,\m R)$ the group of deck transformations. The preimage of $\tau$ by $\cover'$ (see Figure~\ref{fig:covering1}) is a union of disjoint chords in $\m H$ invariant under $\G'$. We now  
    replace each chord with the hyperbolic geodesic in $\m H$ with the same endpoints. In this way we get a $\G'$-invariant family of hyperbolic geodesics. No pair of these  geodesics can cross, because this would lead to a corresponding crossing of the chords arising from $\tau$.  This implies that the images of these geodesics  under  $\cover'$ with the cusps in $\Sigma'$ added give a 
    Jordan  curve $\gamma'$. It is easy to see that this the unique Jordan curve sought after. 
\end{proof}

In \cite[Lem.\,4.2]{MRW1}, it was shown that $\g'$ is different from the unique curve in $\mc L (w_1, \ldots, w_n; \tau)$ with the geodesic property discussed in the rest of the paper, except when $\g'$ is a circle. We now show that their corresponding projective structures are related by a $\pi$-angle grafting.

\begin{figure}[ht]
\centering
  \includegraphics[width=0.8\linewidth]{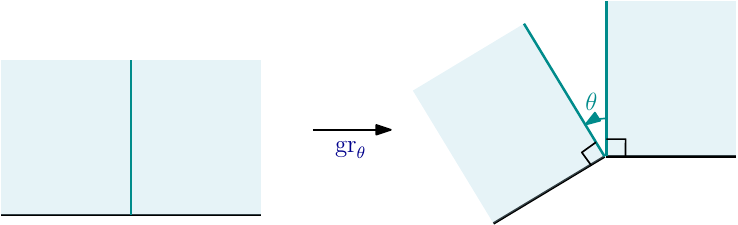}
  \caption{\label{fig:grafting1} Grafting by $\theta$ along $\ii \m R_+$.}
\end{figure} 

\emph{Grafting} is a way to deform a Riemann surface by inserting a new space along a measured lamination (see, e.g., \cite{goldman1987projective,tanigawa1997grafting}).
In the simplest case, grafting by an angle $\theta > 0$ along the imaginary axis in the upper half-plane model can be defined by constructing a Riemann surface where one inserts a wedge of angle $\theta$ at the origin as in Figure~\ref{fig:grafting1}.
Grafting along other geodesics in $\H$ can always be reduced to this case by conjugating by a Möbius transformation. Grafting along a geodesic on a general hyperbolic surface is realized in such a way that it lifts to the grafting along all the lifts of the geodesic on the universal cover $\m H$.
Grafting also defines a projective structure on the resulting surface, since locally the construction takes place on the Riemann sphere and defines a developing map in a natural way.

We will only consider grafting along countably many geodesics by the angle $\theta = \pi$ where the inserted wedge is a half-plane, or in the case where we graft along a semicircle and the inserted region becomes a disc, see Figure~\ref{fig:grafting2}.

\begin{figure}[ht]
\centering
  \includegraphics[width=0.8\linewidth]{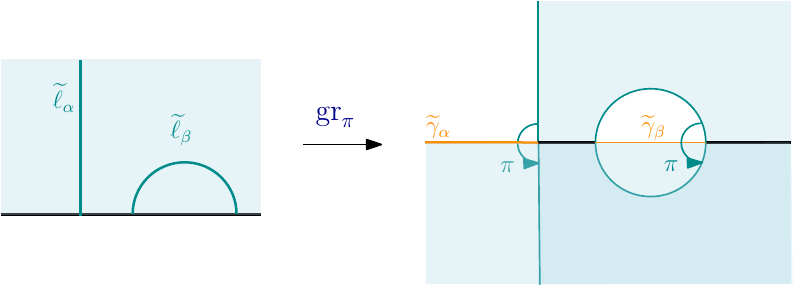}
  \caption{\label{fig:grafting2} Grafting by $\pi$ along two geodesics $\a$ and $\b$.}
\end{figure}

Let $\g'$ be a Jordan curve as in Lemma~\ref{lem:geo_puncture}. Let $\g_k'$ be the geodesic on $\Sigma'$ between $w_k$ and $w_{k+1}$ (which is part of $\g'$) and $\cover' : \m H \to \Sigma'$ be a covering map. Let $\widetilde \S$ be the Riemann surface obtained by grafting along all lifts $\widetilde{\ell}_\alpha$ of the geodesics $\g_k'$ ($k = 1, \dots, n$), inserting a $\pi$-angle wedge $\widetilde V_\alpha$. 
Let $\S = \Chat \smallsetminus \{z_1, \ldots, z_n\}$ be the corresponding $n$-punctured Riemann surface obtained by grafting into each $\g_k'$ a $\pi$-angle wedge $V_k$. This way, we obtain a covering map $\cover : \widetilde \S \to \S$.
We denote by $\widetilde \g_\alpha$ the central line of $\widetilde V_\alpha$ and $\g_k$ the central line of $V_k$, so that the image of $\widetilde \g_\a$ under $\cover$ is one of $\{\g_1, \ldots, \g_n\}$. 
See Figure~\ref{fig:grafting_sphere} illustrating the grafting map on a $4$-punctured sphere.

\begin{figure}[ht]
\centering
  \includegraphics[width=0.9\linewidth]{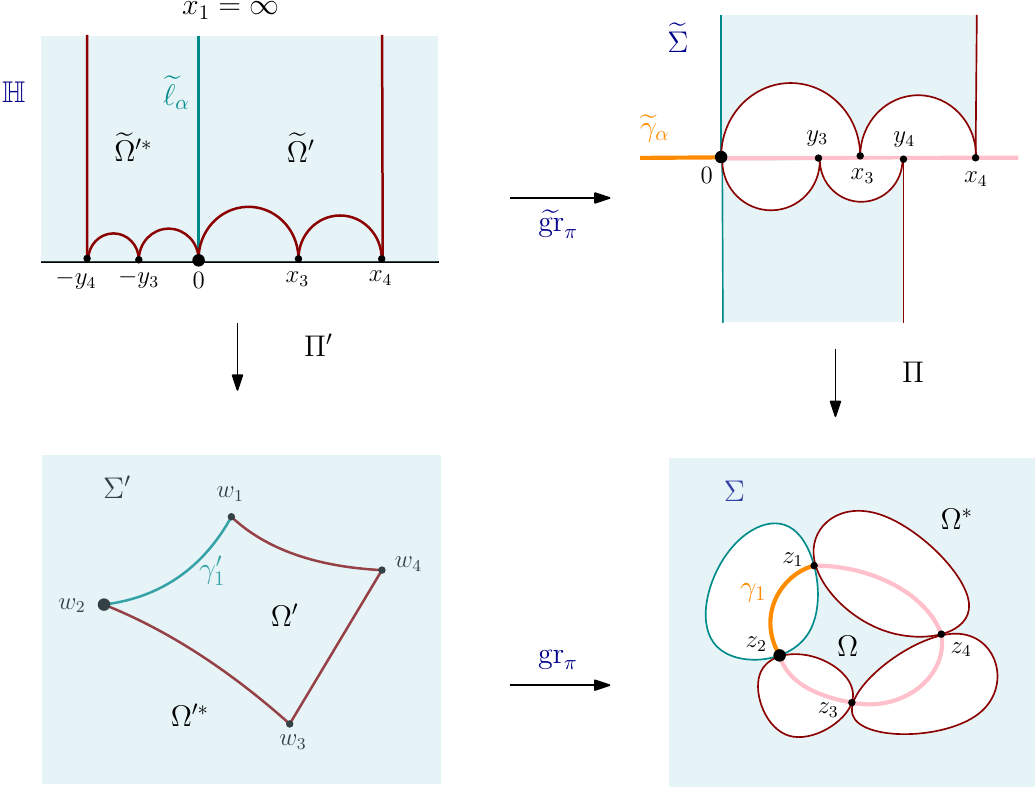}
  \caption{\label{fig:grafting_sphere} Illustration of the grafting by $\pi$ along the geodesics $\g'_k$ on a $4$-punctured sphere $\Sigma'$, for $k = 1, \ldots, 4$, in a local chart. }
\end{figure}

\begin{prop}\label{prop:central_geod}
    The central lines $\widetilde \g_\alpha$ form a graph $\widetilde G$ on $\widetilde \S$ with geodesic property. In particular, this implies that the Jordan curve $\g : = \g_1 \cup \cdots \cup \g_n$  formed by the union of the central lines of $(V_k)_{k = 1,\ldots, n}$ has the geodesic property. Moreover, the welding homeomorphism of $\g$ is $C^1$ smooth, and the inverse of $\cover$ is a projective chart of the projective structure associated with $\g$ as in Section~\ref{sec:projective_from_curve}.
    It has a holonomy representation given by
    \begin{equation}\label{eq:rho_from_rho_F}
        \rho_\gamma(\eta_k) = S_{k-1} \circ \rho_F(\eta_k') \circ S_k,
    \end{equation}
    where $\rho_F$ is a holonomy representation of the uniformizing structure on $\Sigma'$, $\eta_k$ and $\eta_k'$ are the standard generators of $\pi_1(\Sigma)$ and $\pi_1(\Sigma')$, respectively, and $S_k$ is the Möbius transform that rotates by angle $\pi$ around the two fixed points $x_k$ and $x_{k+1}$ of $\rho_F(\eta_k')$ and $\rho_F(\eta_{k+1}')$, respectively. 
\end{prop}

\begin{proof}
  Let $\widetilde \ell_\a \subset \m H$ be a lift of the geodesic $\g_1'$ (the same proof works for $\g_k'$). We now check that the corresponding central line $\widetilde \g_\a$ is the hyperbolic geodesic in the union of the two adjacent faces of $\widetilde G$.  
  Without loss of generality, we assume that $\widetilde \ell_\a  = \ii \m R_+$, and the covering map $\cover'$ maps $\infty$ to $w_1$ and $0$ to $w_2$. 
  A fundamental domain of $\cover'$  is given by a union of two ideal $n$-gons $\widetilde \O'$ and $\widetilde \O'^*$ with $\ii \m R_+$ being one of the edges. See Figure~\ref{fig:grafting_sphere}.   It is immediate that after grafting by $\pi$ along all the edges of $\widetilde \O'$ and $\widetilde \O'^*$, the central line $\widetilde \g_\a$ obtained is $ \m R_-$, and the union of the central lines associated with the other edges of $\widetilde \O'$ (resp. of $\widetilde \O'^*$) is $\m R_+$. Since $\m R_-$ is the hyperbolic geodesic of $\m C \smallsetminus \m R_+$, this shows the geodesic property of $\widetilde G$.

  In the local chart $\m C \smallsetminus \m R_-$ around $\widetilde \g_\a = \m R_-$, the covering map $\cover : \widetilde \S \to \S$ maps respectively the upper and lower halfplanes to the two connected components of $\Chat \smallsetminus \g$. This shows that $\g_1$ is also a hyperbolic geodesic of the simply connected domain $\Chat \smallsetminus (\cup_{k\ge 2} \g_k)$, %
  proving the geodesic property of $\g$.

  Now we show that the welding homeomorphism of $\g$ is $C^1$.
  Notice that in the local chart illustrated in Figure~\ref{fig:grafting_sphere}, we may already read a welding homeomorphism $\varphi$ of $\g$.
  Indeed, we can write $\varphi : = \cover_{-}^{-1} \circ \cover_+|_{\widehat{\m R}}$, where $\cover_{+} = \cover|_{\H}$ and $\cover_{-} = \cover|_{\H^*}$ and in this chart $\cover$ is continuous along $\widetilde \g_\a = \m R_-$, so that $\varphi|_{\m R_-}$ is the identity map.
  We also notice that $\varphi (x_k) = y_k$ for $k=3,\ldots,n.$
  By Proposition~\ref{prop:hol}, $\varphi|_{[0,x_3]}$ is given by the holonomy $\rho_\gamma(\eta_2) \in \PSL(2,\m R)$ around $z_2$, which is of the form $\rho_\gamma(\eta_2) = S_1 \circ \rho_F(\eta_2') \circ S_2$. Here, $\rho_F(\eta_2') \in \PSL(2,\m R)$ is the holonomy around $w_2$ for the Fuchsian projective structure which is parabolic and fixes $0$, $S_1 (z) = - z$ realizes the $\pi$-grafting along $\widetilde \ell_\a = \ii \m R_+$, and $S_2$ realizes the $\pi$-grafting along the geodesic passing through $0$ and $x_3$ (which is a conjugate of $S_1$). In particular,  both $S_1$ and $S_2$ have derivative $-1$ at $0$. We obtain that $\rho(\eta_2)'(0) = 1$. This shows that $\varphi'$ is continuous at $0$.
  Similarly, if $S_k$ realizes the $\pi$-grafting along the geodesic between $x_k$ and $x_{k+1}$, we see that $\rho_\gamma(\eta_k) = S_{k-1} \circ \rho_F(\eta_k) \circ S_{k}$ is parabolic, and by Proposition~\ref{prop:hol} and induction we have
  \begin{equation}\label{eq:varphi_from_rho}
      \varphi|_{[x_k,x_{k+1}]} = \rho_\gamma(\eta_2) \circ \dots \circ \rho_\gamma(\eta_k).
  \end{equation}
  This completes the proof.
\end{proof}

Proposition~\ref{prop:central_geod} shows that from a marked Fuchsian projective structure and its unique geodesic Jordan curve $\g'$, the $\pi$-angle grafting along all the arcs of $\g'$ gives the projective structure associated with a Jordan curve with the geodesic property. Projecting down to $T_{0,n}$, we obtain a map $\operatorname{gr}_\pi : T_{0,n} \to T_{0,n}$.

\begin{cor}\label{cor:bijection}
    The map $\operatorname{gr}_\pi$ is a differentiable bijection $T_{0,n} \to T_{0,n}$.
\end{cor}
\begin{proof}
    We first show that $\operatorname{gr}_\pi$ is bijective by identifying its inverse, namely, a de-grafting map. 

    Let $\mc L (z_1, \ldots, z_n; \tau)$ be an isotopy class of Jordan curves representing a point in $T_{0,n}$. Theorem~\ref{thm:uniqueness} shows that there exists a unique Jordan curve $\g = \g_1 \cup \cdots \cup \g_n$ in $\mc L (z_1, \ldots, z_n; \tau)$ with the geodesic property. Let $\O$ and $\O^*$ denote the two connected components of $\Chat \smallsetminus \g$.
    
    From the construction of the projective structure associated with $\g$ in Section~\ref{sec:projective_from_curve}, there is a projective chart $H : \O \cup \O^* \cup \g_1 \to \m H \cup \m H^* \cup \m R_-$, such that $H(\g_1) = \m R_-$. We note that $H$ coincides with $\cover^{-1}$ in the local chart shown in Figure~\ref{fig:grafting_sphere}. Let $V_1$ denote $H^{-1} (\{z \colon \Re (z) < 0\})$  and similarly $V_k$ for all $k = 2, \ldots, n$. We note that $\partial V_1 \cap \O$ is a hyperbolic geodesic in $\O$. This implies that $V_k$ are pairwise disjoint. Let $V = \cup_{k = 1}^n V_k$.
    
    We let $\Sigma' = \Chat \smallsetminus \{w_1, 
    \ldots, w_n\}$ be the Riemann surface obtained by  identifying $\partial V_k \cap \O$ with $\partial V_k \cap \O^*$ such that in the chart $H$, $ \ii x$ is identified with $-\ii x$. We modify the projective chart by degrafting by $\pi$, so that the new chart is given by
    $$H' (z) = \begin{cases}
        H(z), \quad & z \in \O\smallsetminus V \\
        - H(z), & z \in \O^*\smallsetminus V.
    \end{cases}$$
    In particular, it takes value in $\widetilde \O' \cup \widetilde \O'^* \subset \m H$, where $\widetilde \O'$ and $\widetilde \O'^*$ are illustrated in Figure~\ref{fig:grafting_sphere}. By induction, the analytic continuation $\widetilde H$ of $H$ to the universal cover of $\Sigma'$ takes value in $\m H$ and is injective. The same argument as in the proof of Proposition~\ref{prop:central_geod} shows that the holonomy associated with the projective structure on $\Sigma'$ is parabolic and in $\PSL(2,\m R)$ around each puncture. Therefore, $H$ defines a hyperbolic metric on $\Sigma'$ by pulling back the metric on $\m H$ for which the arcs of $\g'$ are geodesics (as the holonomies are isometries of $\m H$). 
    
    Moreover, the parabolic holonomy around $w_k$ implies that there is a neighborhood of $w_k$ that is isometric to a cusp for all $k$. 
    Hence, $\Sigma'$ is complete. We deduce that $\widetilde H$ is surjective onto $\m H$ and defines a Fuchsian projective structure on $\Sigma'$.

    To see that $\operatorname{gr}_{\pi}$ is differentiable, we note that by \cite[Lemma~3]{TZ1} the (multi-valued) map $(w_1,\dots,w_n) \mapsto \rho_F$ is complex analytic, while $\rho_F \mapsto \rho_\gamma$ is differentiable by the explicit formula \eqref{eq:rho_from_rho_F}, $\rho_\gamma \mapsto \varphi$ is differentiable by formula \eqref{eq:varphi_from_rho} and finally $\varphi \mapsto (z_1,\dots,z_n)$ is differentiable by Theorem~\ref{thm:differentiability}.
\end{proof}

\bigskip
\noindent {\bf Acknowledgments: } 
We thank Peter Lin, Donald Marshall, and Curtis McMullen for helpful discussions and the anonymous referee for the careful reading and valuable comments. M.B.\ is supported by NSF grant DMS-1808856.
J.J.\ was supported by The Finnish Centre of Excellence (CoE) in Randomness and Structures, and is also supported by The Knut and Alice Wallenberg Foundation.
S.R.\ is supported by NSF Grants DMS-1954674 and DMS-2350481. Y.W. is supported by the European Union (ERC, RaConTeich, 101116694)\footnote{Views and opinions expressed are however those of the author(s) only and do not necessarily reflect those of the European Union or the European Research Council Executive Agency. Neither the European Union nor the granting authority can be held responsible for them.}. This work is also supported by NSF Grant DMS-1928930 while the authors participated in a program hosted
by the Mathematical Sciences Research Institute in Berkeley, California, during
the Spring 2022 semester.

\bibliographystyle{abbrv}
\bibliography{ref}

\end{document}